\documentclass[10pt]{article}

\usepackage{amsmath}
\usepackage{manfnt}
\usepackage{amssymb}
\usepackage{graphicx}
\usepackage{geometry}
\usepackage{amsthm}
\usepackage[dvipsnames]{xcolor}
\usepackage[utf8]{inputenc}
\usepackage{pst-node}
\usepackage{tikz-cd} 
\usepackage{hyperref}
\usepackage{bbold}
\usepackage{pdfpages}
\usepackage[backend=biber,doi=false,isbn=false,url=false,maxnames=99]{biblatex}
\newcommand{\opIm}{\operatorname{Im}}
\newcommand{\opker}{\operatorname{ker}}

\newcommand{\oploc}{\operatorname{loc}}

\newcommand{\opsupp}{\operatorname{supp}}

\newcommand{\opRe}{\operatorname{Re}}
\newcommand{\opRes}{\operatorname{Res}}

\newcommand\numberthis{\addtocounter{equation}{1}\tag{\theequation}}

\newtheorem{Theorem}{Theorem}[section]
\newtheorem{Proposition}[Theorem]{Proposition}
\newtheorem{Remark}[Theorem]{Remark}
\newtheorem{Definition}[Theorem]{Definition}
\newtheorem{Lemma}[Theorem]{Lemma}

\newtheorem{Assumption}[Theorem]{Assumption}

\title{Control and stabilization of cascade coupled systems: application to a 1-d heat and wave coupled system}
\author{Lucas Davron, Swann Marx, Pierre Lissy}
\begin{document}
\maketitle	

\begin{abstract}
We study cascade coupled systems, for which our prototypical example is a 1-d heat equation coupled with a 1-d wave equation. The heat component is controlled through one boundary and the information is transmitted through another one to the wave component, while the wave component does not influence the heat component. Our aim is to understand the well-posedness, controllability and stabilizability properties for such a system. Establishing well-posedness is tedious using the classical energy method, which motivates us to take advantage of the cascade structure. Taking again advantage of this structure, we prove a simultaneous exact and approximate controllability result. Finally, we obtain polynomial stabilization by means of a closed-loop control defined through the solution to a Sylvester equation. These results are all discussed in an abstract LTI framework and most of our findings apply to more general situations.
\end{abstract}

\vspace*{0.3cm}


\vspace*{0.3cm}

\section{Introduction}
In this paper we investigate the well-posedness, the controllability and the stabilization of cascade coupled systems. Our prototypical example is
\begin{equation}\label{eq:coupled_heat_wave}
\left\lbrace \begin{array}{rcl cc}
z_t(t,x) &=& z_{xx}(t,x),&t>0,& 0 < x < 1, \\
z_x(t,0) &=&0, \\
z_x(t,1) &=& u(t),\\
w_{tt}(t,x) &=& w_{xx}(t,x),&t>0,& 0 < x < 1,\\
w_x(t,0) &=& z(t,0), \\
w(t,1) &=&0,
\end{array}\right.
\end{equation}
where $u = u(t)$ denotes the control. However, many of the results stated in the present paper might apply to more general situations (see Proposition \ref{prop:identification_control_operator}, Theorem \ref{theo:abstract_control_system}, Propositions \ref{prop:riesz_basis},  \ref{prop:sylv}, \ref{prop:closed_loop} and \ref{prop:non-resonance}, and  Theorem \ref{theo_strong_stbl}). We call it \textit{cascade} because the information goes from the heat component to the wave component, but the wave component does not influence the heat component. This system can be thought of as an overly simplified fluid-structure model, it was introduced in \cite{G22} under a more complicated form to model the stabilization of earthquakes by the localized injection of a fluid in the earth's crust. Systems of coupled heat and wave equations are used to model fluid-structure interactions \cite[Chapter 7, \S 7]{handbook}, but also thermoelasticity \cite[\S 1]{liu}. 
\newline
\newline
The properties of the individual heat and wave systems appearing in \eqref{eq:coupled_heat_wave} are well known. The heat equation (the controller)
\begin{equation}\label{eq:heat}
\left\lbrace \begin{array}{rcl cc}
z_{t}(t,x) &=& z_{xx}(t,x),&t>0,& 0 < x < 1,\\
z_x(t,1) &=& u(t),\\
z_x(t,0) &=&0,
\end{array} \right. 
\end{equation}
with initial condition in $L^2(0,1)$ and control $u$ in $L^2$, is null controllable in arbitrarily small time \cite{fattorini_russell}. Due to the Neumann boundary condition, there is one direction in which the heat equation is not exponentially stable. It is well-known that a stabilizing feedback is \textit{e.g.} the collocated feedback $u(t) = - z(t,0)$, which makes it exponentially stable. 

The wave equation (the plant)
\begin{equation}\label{eq:waves}
\left\lbrace \begin{array}{rcl cc}
w_{tt}(t,x) &=& w_{xx}(t,x),&t>0,& 0 < x < 1,\\
w_x(t,0) &=& v(t),\\
w(t,1) &=& 0, 
\end{array}\right.
\end{equation}
has state space $H^1_{(1)}(0,1)\times L^2(0,1)$, where the subscript ``$(1)$" means ``functions vanishing at $x=1$". When the control $v$ is $L^2$, it is exactly controllable (to any state) in time $T = 2$, which is an easy consequence of the method of the characteristics. When $v(t) \equiv 0$, trajectories have constant norm with respect to time, hence the uncontrolled trajectories do not converge to $0$. The collocated feedback $v(t) = w_t(t,0)$ allows to obtain exponential stability. 

For systems of coupled PDEs, the situation is more involved. A first difficulty is that the coupling may destroy the structure of the individual systems. For instance, \eqref{eq:heat} is parabolic, whereas \eqref{eq:waves} is hyperbolic, but  \eqref{eq:coupled_heat_wave} is neither parabolic nor hyperbolic. 
\subsection{Well-posedness}
The \textit{well-posedness} of a linear coupled system of heat and wave equations is usually obtained by appealing to the Lumer-Phillips theorem. However, the ``standard" energy  of \eqref{eq:coupled_heat_wave} is given by
\[
E(t) := \frac{1}{2} \int_0^1 \left\lbrace |z(t,x)|^2 +|w_t(t,x)|^2 +  |w_x(t,x)|^2 \right\rbrace dx,
\]
which satisfies the dissipation law
\begin{equation}\label{eq:dissipation}
\dot{E}(t) = z(t,1)u(t) - \int_0^1|z_x(t,x)|^2 dx - w_t(t,0)z(t,0).
\end{equation}
In particular, for $u(t) \equiv 0$, one cannot obtain the energy estimate 
\[
\exists c \in \mathbb{R},\quad \forall (z^0,w^0,w_t^0),\quad \forall t \geq 0,\quad 
\dot{E}(t) \leq c E(t),
\]
and the Lumer-Phillips theorem does not apply in the natural energy space. To overcome this difficulty, we  will make the following observation (see Theorem \ref{theo:abstract_control_system}): the cascade coupling of an abstract linear system and an abstract linear control system (as defined in Definition \ref{def:abstract_linear_system}) automatically defines another abstract linear control system. In this sense, the system \eqref{eq:coupled_heat_wave} is well-posed, and we will show that its infinitesimal generator admits a closed formula. Subtleties may however appear when one tries to identify the control operator (see \S \ref{sec:WP}). We rewrite \eqref{eq:coupled_heat_wave} in the form of an \textit{abstract cascade coupled system}
\begin{equation}\label{eq:cascade}
\left\lbrace \begin{array}{rcl}
\dot{z} &=& Az + Bu,\\
\dot{\mathsf{w}} &=& E\mathsf{w} + FCz,
\end{array} \right.
\end{equation}
where $(A,B,C)$ represents the heat system \eqref{eq:heat} (the controller), $(E,F)$ represents the wave system \eqref{eq:waves} (the plant) and $\mathsf{w}=(w,w_t)$ is the wave component\footnote{See \cite{LTI,tucsnak2009observation,Coron} for the theory of linear and time-invariant systems.}. Consider the state variable 
\[
\mathbf{Z} = \left( \begin{array}{c}
z \\
w \\
\tilde w
\end{array} \right),
\]
where $\tilde w$ represents $w_t$, and the state space
\[ \mathcal{X} := L^2(0,1) \times H^1_{(1)}(0,1) \times L^2(0,1) ,\]
 together with the input space $\mathcal{U} = \mathbb{C}$. Let $\mathcal{A}$ be the closed and densely defined unbounded operator on $\mathcal{X}$ defined  by
\[
\mathcal{A} = \left( \begin{array}{ccc}
\partial_{xx} & 0 & 0 \\
0 & 0 & 1 \\
0 & \partial_{xx} & 0
\end{array} \right),\]
on the domain
\begin{equation}\label{eq:primal_domain}
D(\mathcal{A}) = \left\lbrace \left( \begin{array}{c}
z \\ w \\ \tilde w
\end{array} \right) \in H^2(0,1) \times H^2(0,1) \times H^1(0,1) : \left\vert \begin{array}{ccccc}
z_x(0) &=& z_x(1) &=& 0, \\
w(1) &=&\tilde w(1) &=&0, \\
w_x(0) &=& z(0).
\end{array}\right. \right\rbrace.
\end{equation}
Standard computations show that 
\[
\mathcal{A}^* = \left( \begin{array}{ccc}
\partial_{xx} & 0 & 0 \\
0 & 0 & -1 \\
0 & -\partial_{xx} & 0
\end{array} \right),\]
with domain (notice that the cascade is reversed)
\begin{equation}\label{eq:adj_coupled}
D(\mathcal{A^*}) = \left\lbrace \left( \begin{array}{c}
\varphi \\ \psi \\ \tilde \psi 
\end{array} \right) \in H^2(0,1) \times H^2(0,1) \times H^1(0,1) : \left\vert \begin{array}{ccccc}
\psi(1) &=& \tilde \psi(1) &=& 0, \\
&& \psi_x(0) &=& 0,\\
&& \varphi_x(1) &=& 0, \\
&& \varphi_x(0) &=& \tilde \psi(0).
\end{array}\right. \right\rbrace.
\end{equation}
Define the  control operator $\mathcal{B} : \mathcal{U} \rightarrow D(\mathcal{A}^*)'$ by\footnote{The space $D(\mathcal{A}^*)'$ is the dual of $D(\mathcal{A}^*)$ with respect to the pivot $\mathcal{X}$, often denoted $\mathcal{X}_{-1}$, see \cite[\S 2.9]{tucsnak2009observation} and \cite[Chapitre III]{aubin}.}
\[
\forall \left( \begin{array}{c}
\varphi \\
\psi \\
\tilde{\psi}
\end{array} \right) \in D(\mathcal{A}^*),\quad 
\mathcal{B}^* \left( \begin{array}{c}
\varphi \\
\psi \\
\tilde{\psi}
\end{array} \right) = \varphi(1).
\] 
As for most of the ensuing results, the following Proposition will be proved in a more abstract framework (see Proposition \ref{prop:identification_control_operator}), we give here its counterpart for \eqref{eq:coupled_heat_wave}.
\begin{Proposition}\label{prop:wp}
The operator $\mathcal{A}$ generates a $C_0$-semigroup on $\mathcal{X}$, and $\mathcal{B}$ is an admissible control operator for the semigroup generated by $\mathcal{A}$. 
\end{Proposition}
For the ``coupled" operator $\mathcal{A}$, a Riesz spectral structure can be deduced from that of $A$ and $E$. Indeed, in general it is plain that $\sigma_p(\mathcal{A}) \subset \sigma_p(A) \cup \sigma_p(E)$, and the inclusion is an equality whenever $A$ and $E$ share no common eigenvalue. From this, we will provide a sufficient condition for $\mathcal{A}$ to be diagonalizable in a Riesz basis, its eigenvectors being computed from the eigenvectors of $A$ and $E$. Below is the application to \eqref{eq:coupled_heat_wave}, see Proposition \ref{prop:riesz_basis} for a more general result. 
\begin{Proposition}\label{prop:riesz_heat_wave}
The point spectrum of $\mathcal{A}$ satisfies 
\[
\sigma_p(\mathcal{A}) = \{ -\lambda_j : j=1,2,... \} \cup \{ i \mu_k : k = 0 , \pm 1 ,... \},\quad \lambda_j = ((j-1)\pi)^2,\quad \mu_k = \left(k + \frac{1}{2} \right)\pi.
\]
The operator $\mathcal{A}$ is diagonalizable in a Riesz basis of simple eigenvectors. We denote  $\mathbf{Z}_j^p$ (resp. $\mathbf{Z}_k^h$) a corresponding normalized eigenvector for $-\lambda_j$ (resp. $i\mu_k$). 
\end{Proposition}
\subsection{Controllability}
The \textit{controllability} properties of a coupled heat-wave system depend on the structure of the coupling. For the thermoelasticity, one of the first control result we are aware of is due to Hansen \cite{hansen} in dimension $d=1$. He obtains the null-controllability of the full state by obtaining an Ingham inequality, which can be thought of a superposition of the observability inequalities of the wave and the heat equation. This idea has been used in \cite{Zhang_Zuazua}, where it was moreover shown that their system is null-controllable for very smooth initial data, while initial data in $D(\mathcal{A}^N)$ cannot be steered to rest by $(H^N(0,T))'$ controls,\footnote{This fact was informally stated in \cite{Zhang_Zuazua}, after the proof of their Theorem 4.2, with $H^{-s}(0,T)$ in place of $(H^N(0,T))'$. However, as noticed in \cite[\S 4.1.1]{davron_rough}, LTI systems are not necessarily well-posed when the control is in $H^{-1}(0,T;U)$. They are well-posed with controls in $(H^N(0,T;U))'$.} for all $N,T > 0$. In dimension $d \geq 2$, Zuazua obtained a mixed controllability property in \cite{zuazua_control_thermoelastic}, where the wave component is exactly controlled and the heat component is approximately controlled. Lebeau and Zuazua \cite{lebeau_zuazua} obtained null controllability of both components, for a simplified thermoelastic model. 

A standard procedure to obtain control results for systems of PDEs is to use the duality between controllability and observability, in order to derive approximate, null, or exact controllability. Here, the control pair $(\mathcal{A},\mathcal{B})$ of the system \eqref{eq:coupled_heat_wave} enjoys the following additional properties: $\mathcal{A}$ is diagonalizable and $\mathcal{B}$ is a scalar action. For such a pair, the observability inequality which characterizes the null controllability takes the form of a so-called Ingham type inequality for scalar non-harmonic Fourier series \cite{komornik_loreti}. The key idea of \cite{hansen} is to start from two Ingham inequalities, already known for the sub-systems, and to deduce a new Ingham inequality implying the observability inequality. This idea has been generalized in \cite[Proposition 1.6]{chowdhury}, whose range of applicability contains our system \eqref{eq:coupled_heat_wave}. The following result will be proved in Appendix \ref{sec:app_proof_zz}.
\begin{Proposition}\label{prop:zz_pour_nous}
The system \eqref{eq:coupled_heat_wave} satisfies the following controllability properties: 
\begin{itemize}
\item It is approximately controllable in time $T = 2$, and is not approximately controllable in any time $T < 2$. 
\item It is not null-controllable in any time $T > 0$, even when taking initial conditions in $D(\mathcal{A}^N)$ and controls in $(H^N(0,T))'$, for fixed $N \in \mathbb{N}$.
\item For all time $T > 2$, it is null-controllable for initial conditions of the form
\[
\mathbf{Z}^0 = \sum_{j=0}^{+\infty} \alpha_j \mathbf{Z}_j^p + \sum_{k=-\infty}^{+\infty} \beta_k \mathbf{Z}_k^h,\quad  \sum_{j=0}^\infty |\alpha_j|^2 +  \sum_{k=-\infty}^{+\infty} \left| \beta_k \sqrt{1+|k|}e^{\sqrt{\pi |k|/2}} \right|^2   < \infty.
\]
\end{itemize}
\end{Proposition}
This Proposition is very similar to \cite[Theorem 4.3]{Zhang_Zuazua}, we only state it for completeness. To go beyond this result, we use recent results on hybrid (or mixed) range inclusions of operators \cite{lissy_douglas} to deduce sufficient conditions for a hybrid controllability property of the abstract system \eqref{eq:cascade}. For the concrete system \eqref{eq:coupled_heat_wave} we shall deduce the following. 
\begin{Theorem}\label{theo:mixed_quakes}
Let $T \geq 2$. Then, for all $\mathbf{Z}^0 \in \mathcal{X}$, $(w^T,\tilde w^T) \in H^1_{(1)}(0,1) \times L^2(0,1)$ and $\epsilon > 0$, there exists $u \in L^2(0,T)$ such that the solution $(z,w,w_t)$ of \eqref{eq:coupled_heat_wave} satisfies
\begin{equation}\label{eq:contro_simul}
z(T) = 0,\quad \mathrm{and} \quad \left\| \left( \begin{array}{c}
    w(T)  \\
    w_t(T) 
\end{array}\right) - \left( \begin{array}{c}
    w^T  \\
    \tilde{w}^T
\end{array}\right) \right\|_{H^1_{(1)}(0,1) \times L^2(0,1)} < \epsilon.
\end{equation}
\end{Theorem}
Mixed controllability has been considered in \cite{zuazua_control_thermoelastic}, where the wave component is exactly controlled and the heat component is approximately controlled. Therein, a decoupling operator transforms the original problem into a cascade coupled system, the wave component being the controller and the heat component being the plant. Moreover, this proof technique is different from ours, as they use compactness-uniqueness arguments while we only rely on Hölmgren's uniqueness theorem. 
\begin{Remark}
The reader may notice that our system is similar to the ones of \cite{Zhang_Zuazua, batty_paunonen,LhachemiPrieurTrelat2025}. The systems studied in the first two cited works behave differently from ours, as their coupling is an interconnection rather than a cascade. More precisely, in these papers, the heat component acts on the wave component, and vice versa. The model studied in the third cited work \cite{LhachemiPrieurTrelat2025} is coupled in cascade, hence seems to be closer to our model. However, their coupling is internal while ours is at the boundary, and they act on the wave component while we act on the heat component. Moreover, their results are of a different nature from ours. In particular, 
they prove exponential stabilizability, which is unlikely to hold in our case, as \eqref{eq:coupled_heat_wave} is not open-loop stabilizable (see the next \S). 
\end{Remark}
\subsection{Stabilization}\label{sec:intro_stbl}
We will consider three stability concepts for a $C_0$ semigroup, say $e^{t \mathcal{A}}$ on the state space $\mathcal{X}$. 
\begin{itemize}
    \item The semigroup $e^{t \mathcal{A}}$ is \textit{strongly stable} if 
    \[
    \forall \mathbf{Z}^0 \in \mathcal{X},\quad e^{t \mathcal{A}} \mathbf{Z}^0 \xrightarrow[t \to \infty]{\mathcal{X}} 0. 
    \]
    \item The semigroup $e^{t \mathcal{A}}$ is \textit{exponentially stable} if 
    \[
    \exists \lambda,c > 0,\quad \forall \mathbf{Z}^0 \in \mathcal{X},\quad \forall t > 0,\quad \|e^{t \mathcal{A}}\mathbf{Z}^0 \|_\mathcal{X} \leq c e^{-\lambda t} \|\mathbf{Z}^0 \|_\mathcal{X}.
    \]
    \item The semigroup $e^{t \mathcal{A}}$ is \textit{non-uniformly stable} (at the rate $m(\cdot)$) if 
    \[
    \forall \mathbf{Z}^0 \in D(\mathcal{A}),\quad \forall t > 0,\quad \|e^{t \mathcal{A}}\mathbf{Z}^0 \|_\mathcal{X} \leq m(t) \|\mathbf{Z}^0 \|_{D(\mathcal{A})}.
    \]
\end{itemize}
We think of strong stability as the lack of decay rate, while exponentially stability is the strongest notion of the above list. We emphasize that in the definition of non-uniform stability, one cannot replace the $D(\mathcal{A})$-norm by the $\mathcal{X}$-norm, without imposing exponential stability. Indeed, recall that for a semigroup the property 
\[
\exists t_0 > 0,\quad \| e^{t_0 \mathcal{A}} \|_{\mathcal{X} \to \mathcal{X}} < 1,
\]
is enough to impose exponential stability. When $m(t) = t^{-\alpha}$ it is common to speak of polynomial stability, see \cite{batty_duyckaerts,batty_tomilov,rozendaal} and the references therein for more on non-uniform stability.
\newline
\newline
The question of the \textit{stability} of coupled heat-wave systems has attracted considerable attention since the seminal result of Dafermos \cite[Theorem 5.2]{dafermos}, imposing strong stability (towards a not necessarily zero state) for systems of linear thermoelasticity. Many contributions have then tried to obtain decay rates. In the one dimensional case, an exponential rate was obtained independently by energy estimates and spectral theory \cite{slemrod,hansen_stbl}. In dimension $d=3$, Lebeau and Zuazua \cite{lebeau_zuazua_stbl} gave a mild sufficient condition of geometric nature for the impossibility of exponential stability. Similar results exist for systems of coupled heat and wave equations. In dimension $d \geq 2$, Zhang and Zuazua \cite{zhang_zuazua_arma} obtained sharp polynomial and logarithmic decay rates, under very general geometrical assumptions. A polynomial decay rate was obtained for a one-dimensional model in \cite{Zhang_Zuazua} by means of spectral analysis. The sharpness of the decay rate can be obtained by direct verification on a well-chosen solution (as in \cite{zhang_zuazua_arma}), or by taking advantage of the relation between the growth of the norm of the resolvent on the imaginary axis and the decay rate of the associated solutions \cite{batty_tomilov} (as in \cite{avalos_lasiecka_heat_wave, batty_paunonen}). 

Surprisingly the question of \textit{stabilization} has attracted less attention. By stabilization we mean that one aims at improving the decay rate of the solutions, by adding a feedback law, which is a control depending on the state variables (in our prototypical example, the solutions to the wave and the heat equations). For thermoelasticity, exponential stabilization was achieved, for instance, by means of boundary velocity feedback in \cite{liu} and by means of internal damping in \cite{benabd}. Systems of non-linear fluid-structure models are stabilized by non-linear damping in \cite{lasiecka_stbl_damping} and by interface feedback in \cite{lasiecka_stbl_interface}. The coupled heat-wave model studied in \cite{Zhang_Zuazua} is stabilized at exponential rate by dynamic feedback in \cite{zheng}. 
\newline
\newline
The semigroup induced by \eqref{eq:coupled_heat_wave} is bounded, and for any $k \in \mathbb{Z}$ the trajectory starting from $\mathbf{Z}_k^h$ remains with norm $1$, hence strong stability cannot hold without stabilization. We aim at stabilizing the system in closed loop, which is a nontrivial task for at least two reasons. Firstly, from the dissipation law \eqref{eq:dissipation} it is not clear which feedback one should put in \eqref{eq:coupled_heat_wave}. Moreover, the collocated feedback $u(t) = -B^*z(t)$ cannot work because it does not depend on the wave component\footnote{Put $u(t) = -B^*z(t)$ and $z^0(x) \equiv 0$, then $z(t,x) \equiv 0$ and the wave equation is not damped.}. Secondly, the system \eqref{eq:coupled_heat_wave} is not open-loop stabilizable in the sense of \cite{rebarber_zwart}. More precisely, invoking \cite[Theorem 4.5]{rebarber_zwart}, Proposition \ref{prop:riesz_heat_wave} and \eqref{eq:estimate_obs_hyp}, we deduce
\begin{equation}\label{barb}
\forall \epsilon > 0,\quad \exists \mathbf{Z}^0 \in \mathcal{X}	,\quad \forall u \in L^2_{\oploc}[0,\infty),\quad \limsup_{t \to \infty} e^{\epsilon t} \|  \mathbf{Z}(t) \|_\mathcal{X} = \infty.
\end{equation}
It is reasonable\footnote{We do not know if feedback stabilization in the sense of Definition \ref{deff} (at a uniform exponential rate) implies open-loop stabilization in the sense of \cite{rebarber_zwart}. However, if $\mathcal{K}$ is an admissible feedback for $(\mathcal{A}, \mathcal{B})$, meaning in particular that the output $\mathcal{K} \mathbf{Z}(t)$ is admissible for the trajectories of $\mathcal{A} + \mathcal{B}\mathcal{K}$, then it is clear that the latter implication holds true. See also \cite{rebarber_weiss, jacob_zwart}.} to conjecture the impossibility of feedback stabilization of \eqref{eq:coupled_heat_wave}, in the natural energy space $\mathcal{X}$, at an exponential (uniform) rate. This suggests to consider non-uniform stabilization. 

We will revisit a classical technique used to stabilize cascade coupled systems of ODEs, in the context of infinite dimensional systems. We will consider a solution $\Pi$ (if it exists) of the Sylvester equation\footnote{\textit{Strico sensu}, it is not totally rigorous; the heat equation with Neumann boundary conditions is not exponentially stable, so we will need to pre-stabilize it and to change $E$ in a stable version $E_\alpha$, see Section \ref{ssub:pol}.} 
\begin{equation}\label{eq:sylvester}
E \Pi = \Pi A + FC,
\end{equation}
so that the change of coordinates
\begin{equation}\label{eq:change_variable}
p = \mathsf{w} + \Pi z, \quad z=z,
\end{equation}
transforms \eqref{eq:coupled_heat_wave} into 
\begin{equation}\label{eq:p_w}
\left\lbrace \begin{array}{rcl}
\dot{z}(t) &=& Az(t) + Bu(t), \\
\dot{p}(t) &=& Ep(t) + \Pi Bu(t).
\end{array}\right.
\end{equation}
Recall that \eqref{eq:coupled_heat_wave} is coupled in cascade: the control $u(t)$ acts on $z(t)$, which acts itself on $\mathsf{w}(t)$. The new system \eqref{eq:p_w} is simultaneously controlled: the control $u(t)$ acts on $z(t)$ and on $p(t)$. This suggests the following strategy:
\begin{enumerate}
\item Show that the Sylvester equation \eqref{eq:sylvester} has a solution $\Pi$;
\item Stabilize $(E,\Pi B)$ in closed-loop with a feedback $K$;
\item Show that $u(t) := Kp(t) \rightarrow 0$ as $t \rightarrow \infty$, and also $z(t) \rightarrow 0$, so that the joint state $(p(t),z(t))$ goes to 0;
\item Transfer this stability property back to the coordinates $(z,\mathsf{w})$ of the original system \eqref{eq:coupled_heat_wave}.
\end{enumerate}
This approach has been successfully generalized to the case where one of the two systems is finite dimensional \cite{natarajan,paunonen2015controller}. The case where both systems are infinite dimensional is more technical, in the works we are aware of, one often makes several extra assumptions such as the boundedness of one of the control or observation operators, see \cite{paunonen2017robust,sylv_PDEPDE,emirsajlow}. In this work we are able to deal with the general case where all the control and observation terms are unbounded. Under the assumption that $A$ is exponentially stable, $E$ is skew-adjoint and other mild technical assumptions, we will show that the Sylvester equation \eqref{eq:sylvester} has a solution $\Pi \in \mathcal{L}_c(Z,W)$(Proposition \ref{prop:sylv}), and the feedback law $u(t) = -(\Pi B)^* p(t)$ achieves strong stability (Theorem \ref{theo_strong_stbl}). As an application, we use the latter feedback law to stabilize the system \eqref{eq:heat} at the polynomial rate $1/\sqrt{1+t}$.
\begin{Theorem}\label{theo:stabl_heat_wave}
There exists a feedback $\mathcal{K} : D(\mathcal{K}) \subset \mathcal{X} \rightarrow \mathbb{C}$ such that the feedback law  $u(t) = \mathcal{K}(z(t),w(t),w_t(t))$ induces a $C_0$-semigroup on the state space $\mathcal{X}$, and denote by $\mathcal{A}_\mathcal{K}$ its generator. Moreover, 
\[
\exists c > 0,\quad \forall (z^0,w^0,\tilde{w}^0) \in D(\mathcal{A}_\mathcal{K}),\quad \forall t > 0,\quad \|(z(t),w(t),w_t(t))\|_\mathcal{X} \leq \frac{c}{\sqrt{1+t}} \|(z^0,w^0,\tilde{w}^0) \|_{D(\mathcal{A}_\mathcal{K})},
\]
where $(z(t),w(t),w_t(t))$ is the solution of \eqref{eq:coupled_heat_wave}. 
\end{Theorem}
The rest of the paper goes as follows. In Section 2, we consider abstract cascade coupled systems, for which we obtain well-posedness and Riesz basis properties, and we apply these abstract results to the heat-wave system \eqref{eq:coupled_heat_wave}, in order to obtain a mixed controllability result. In Section 4, we study the Sylvester equation and the stabilizability properties of the target system, in order to obtain the polynomial stabilization of the heat-wave system \eqref{eq:coupled_heat_wave}. 

\section*{Acknowledgements}
Pierre Lissy was funded by the French Agence Nationale de la Recherche (Grants ANR-22-CPJ2-0138-01 and ANR-20-CE40-0009). Swann Marx was funded by French Agence Nationale de la Recherche under the ROTATION project (grant no. ANR-24-CE48-0759).

\section{Considerations on abstract cascade coupled systems}
In this section, we consider abstract cascade coupled systems, which are formally defined by \eqref{eq:cascade}. We will use the formalism of linear and time-invariant systems, for which we refer to \cite{tucsnak2009observation,LTI} and \cite[\S 2.3]{Coron}. For simplicity, as soon as we consider abstract cascade coupled systems, we refer to the second coordinate of \eqref{eq:cascade} by $w(t)$, there should be no confusion with the state $\mathsf{w}(t) = (w(t),w_t(t))$ of the wave system \eqref{eq:waves}. 
\subsection{Abstract linear systems}
In this \S, we introduce the relevant definitions and collect some background material on abstract linear systems. We begin by motivating the use of this concept with a non-rigorous discussion: in practice, one may often represent a control system by the equations 
\begin{equation}\label{eq:abstract_representation}
    \left\lbrace \begin{array}{rcl}
\dot{x}(t) &=& Ax(t) + Bu(t), \\
x(0) &=& x^0, \\
y(t) &=& Cx(t) + Du(t),
\end{array}\right.
\end{equation}
where $(A,B,C,D)$ are (possibly unbounded) operators, $x(t) \in X$ is the state of the system (at time $t$), $x^0 \in X$ is the initial condition, $u(t) \in U$ is the control and $y(t) \in Y$ is the output signal ($X,Y,U$ are Hilbert spaces). For the systems we have in mind, there holds $D = 0$, but we will see that in general one should take $D$ into account in order to establish a complete theory. Several strategies may be used to derive the well-posedness of \eqref{eq:abstract_representation}, the specific methodology possibly depending on the equations. If the system is well-posed, in a sense that we do not make precise for the moment, one expects that 
\begin{itemize}
    \item $A$ generates a $C_0$-semigroup, denoted $(\mathbb{T}_t)_{t \geq 0}$;
    \item The Duhamel formula holds: 
    \[
    x(t) = \mathbb{T}_t x^0 + \int_0^t \mathbb{T}_{t-\sigma} Bu(\sigma) d\sigma;
    \]
    \item The output is given by the formula 
    \[
    y(t) = C \mathbb{T}_t x^0 + C \int_0^t \mathbb{T}_{t-\sigma} Bu(\sigma) d\sigma + Du(t). 
    \]
\end{itemize}
Then, the equations \eqref{eq:abstract_representation} induce a family of linear continuous operators $((\mathbb{T}_t)_{t \geq 0},(\Phi_t)_{t \geq 0}, \mathbb{L}, \mathbb{F})$, satisfying 
\[
\Phi_t : \left\lbrace \begin{array}{ccc}
    L^2_{\mathrm{loc}}([0,\infty);U) & \longrightarrow & X  \\
    u & \longmapsto &  \int_0^t \mathbb{T}_{t-\sigma} Bu(\sigma) d\sigma
\end{array}\right. \quad 
\mathbb{L} : \left\lbrace \begin{array}{ccc}
    X & \longrightarrow & L_{\oploc}^2([0,\infty);Y)  \\
    x^0 & \longmapsto &  [t \mapsto C \mathbb{T}_t x^0]
\end{array}\right.
\]
and 
\[
\mathbb{F} : \left\lbrace \begin{array}{ccc}
    L_{\oploc}^2([0,\infty);U) & \longrightarrow & L_{\oploc}^2([0,\infty);Y)  \\
    u & \longmapsto &  [t \mapsto C \Phi_t u+Du]
\end{array}\right.
\]
These operators allow one to compute the state $x(\cdot)$ and the output $y(\cdot)$, given the data $(u(\cdot),x^0)$. By analogy with semigroup theory, if the equations \eqref{eq:abstract_representation} corresponds to the abstract ODE $\dot{x}(t) = Ax(t)$, then the family of operators $((\mathbb{T}_t)_{t \geq 0},(\Phi_t)_{t \geq 0}, \mathbb{L}, \mathbb{F})$ corresponds to the concept of $C_0$-semigroup\footnote{A more precise link with the usual notion of semigroup is made thanks the Lax-Philips semigroup, see \cite[Chapter 2.7]{staffans2006well}.}. Moreover, in order to develop the theory, it is easier to start with a family of operators and to then associate it a representation in the state variable as in \eqref{eq:abstract_representation}. We will see that the converse operation is more difficult. 
\begin{Definition}\label{def:abstract_linear_system} Let $X,U,Y$ be Hilbert spaces. An abstract linear system on the state space $X$, the input space $U$, and the output space $Y$ is a quadruple $((\mathbb{T}_t)_{t \geq 0},(\Phi_t)_{t \geq 0},\mathbb{L}, \mathbb{F})$ such that 
\begin{itemize}
    \item[1.] $(\mathbb T_t)_{t\geq 0}$ is a strongly continuous semigroup of operators on $X$.
    \item[2.] For all $t \geq 0$, $\Phi_t : L^2_{\mathrm{loc}}([0,\infty);U) \rightarrow X$ is bounded, and we have 
    \begin{equation}\label{eq:composition}
\Phi_{\tau+t}(u \diamond_\tau v)
= \mathbb T_t \Phi_\tau u + \Phi_t v,\quad \forall t,\tau \geq 0,\quad \forall u,v \in L^2_{\oploc}([0,\infty);U).
\end{equation}
    \item[3.] $\mathbb{L} : X \rightarrow L_{\oploc}^2([0,\infty);Y)$ is bounded and such that for all $x^0 \in X$, all $\tau \geq 0$, we have for almost every $t \geq 0$
    \begin{equation*}
(\mathbb{L}x^0)(t+\tau)
= \left( \mathbb{L}x^0 \diamond_\tau {\mathbb L} \mathbb T_\tau x^0\right)(t).
\end{equation*}
    \item[4.] $\mathbb{F} : L^2_{\oploc}([0,\infty);U) \rightarrow L_{\oploc}^2([0,\infty);Y)$ is bounded and such that for all $u \in L^2_{\oploc}([0,\infty);U)$ and almost every $t \geq 0$, 
    \begin{equation*}
\mathbb{F}(u \diamond_\tau v)(t+\tau) = \left( \mathbb{F} u \diamond_\tau ( \mathbb{L}\Phi_\tau u + \mathbb{F} v)\right)(t).
\end{equation*}
\end{itemize}
An abstract linear control system is a pair $((\mathbb{T}_t)_{t \geq 0},(\Phi_t)_{t \geq 0})$ that satisfies $1.$ and $2.$ above. 
\end{Definition}
In the above, we have used the notation
\[
u \diamond_\tau v(\sigma) = \left\lbrace \begin{array}{lc}
    u(\sigma), & 0 \leq \sigma < \tau,  \\
    v(\sigma-\tau), & \tau \leq \sigma < \infty. 
\end{array}\right.
\]
If one takes $t=\tau = 0$ and $u = v$ in \eqref{eq:composition}, one sees that $\Phi_0 = 0$. Taking $t = 0$ and $v = 0$, one sees that $\Phi_\tau u$ only depends on the restriction of $u$ to the time interval $(0,\tau)$. Therefore, there is a canonical way of interpreting $\Phi_\tau$ as a bounded operator $L^2(0,\tau;U) \to X$. One can similarly show that for all $\tau \geq 0$, the operator $\mathbb{F}$ can be seen as bounded $L^2(0,\tau;U) \to L^2(0,\tau;Y)$. 
\newline
\newline
Given an abstract linear system $\Sigma$, its representation in the state variable can be done as follows. We put $A$ the infinitesimal generator of $\mathbb{T}$. It can be shown \cite[Theorem 3.9]{weiss_admissibility_control} that there exists a bounded operator $B : U \to D(A^*)'$ defined by 
\begin{equation}\label{eq:limit_control}
\forall \mathrm{v} \in U,\quad B\mathrm{v} = \lim_{t \to 0^+} \frac{1}{t} \Phi_t (1 \otimes  \mathrm{v}),    
\end{equation}
where the limit is in $D(A^*)'$ and $1 \otimes  \mathrm{v}$ is the function constant to $\mathrm{v}$. Then, we have 
\begin{equation}\label{eq:repr_Phi}
\Phi_t u = \int_0^t \mathbb{T}_{t-\sigma}Bu(\sigma)d\sigma,\quad \forall u \in L^2_{\oploc}([0,\infty);U),\quad \forall t \geq 0.     
\end{equation}
The condition \eqref{eq:repr_Phi} implies \eqref{eq:limit_control}, which uniquely determines $B$. It is called the \textit{control operator} of $\Sigma$. There is moreover a unique operator $C : D(A) \to Y$ such that 
\[
\forall x^0 \in D(A),\quad \forall t \geq 0,\quad (\mathbb{L}x^0)(t) = C \mathbb{T}_t x^0,
\]
called the \textit{observation operator} of $\Sigma$. Without further hypothesis on $\Sigma$, its state space representation cannot be much improved. One difficulty is to define the quantity $C \Phi_t u$, as $\Phi_t$ ranges in $X$ and $C$ has domain $D(A)$. To overcome this, it was suggested by Weiss to consider \textit{regular} abstract linear systems \cite{weiss_representation,weiss1994transfer}. 
\begin{Definition}
Let $((\mathbb{T}_t)_{t \geq 0},(\Phi_t)_{t \geq 0},\mathbb{L}, \mathbb{F})$ be an abstract linear system on $X,U,Y$. The system is called regular if, for all $\mathrm{v} \in U$, the following limit exists in $Y$: 
\begin{equation}\label{regularity}
D \mathrm{v} := \lim_{t \to 0^+} \frac{1}{t} \int_0^t \mathbb{F} (1 \otimes \mathrm{v})(\sigma)d\sigma. 
\end{equation}
\end{Definition}
The above (uniquely) defines a bounded operator $D : U \to Y$, called the \textit{feedthrough operator} of $\Sigma$. It can be shown  (see \cite[Theorem 4.5]{weiss_representation}) that when $\Sigma$ is regular, the operator $C$ has an extension $C_L : D(C_L) \subset X \to Y$, called the Lebesgue extension, such that for any $u \in L^2_{\oploc}([0,\infty);U)$, for almost every $t \geq 0$, we have $\Phi_t u \in D(C_L)$ and
\[
(\mathbb{F}u)(t) = C_L \Phi_t u + Du(t). 
\]
Then, we have the following.
\begin{Theorem}{\cite[Theorem 4.6]{weiss_representation}}
Let $\Sigma = ((\mathbb{T}_t)_{t \geq 0},(\Phi_t)_{t \geq 0},\mathbb{L}, \mathbb{F})$ be a regular abstract linear system on the state space $X$, the input space $U$ and the output space $Y$. Denote by $A,B,C,D$ respectively the infinitesimal generator of $\mathbb{T}$, the control operator of $\Sigma$, the observation operator of $\Sigma$, and the feedthrough operator of $\Sigma$. Then for all $x^0 \in X$ and $u \in L^2_{\oploc}([0,\infty);U)$, the functions 
\[
x(t) := \mathbb{T}_t x^0 + \Phi_t u,\quad y(t) := (\mathbb{L}x^0)(t) + (\mathbb{F}u)(t),
\]
satisfy the following 
\begin{itemize}
    \item The curve $x(\cdot)$ is of class $C([0,\infty);X)$ and is the unique solution (by transposition\footnote{For the concept of transposition solution, see \cite[\S 2.3.1]{Coron}. }) of the problem 
    \[
    \dot{x}(t) = Ax(t) + Bu(t),\quad x(0) = x^0.
    \]
    \item For almost every $t \geq 0$ we have $x(t) \in D(C_L)$ and 
    \[
    y(t) = C_L x(t) + Du(t). 
    \]
\end{itemize}
\end{Theorem}
Accordingly, we call $(A,B,C,D)$ the generating operators of $\Sigma$. As anticipated, the converse operation of starting from operators $(A,B,C,D)$ and deciding whether they are the generating operators of some abstract linear system $\Sigma$ is tedious. For abstract linear control systems the situation is completely understood \cite{weiss_representation}: a pair $(A,B)$ generates an abstract linear control system if and only if $A$ generates a $C_0$-semigroup $\mathbb{T}$ and $B \in \mathcal{L}_c(U;D(A^*)')$ is admissible, \textit{i.e.}
\[
\exists \tau > 0,\quad \forall u \in L^2(0,\tau;U),\quad \int_0^\tau \mathbb{T}_{\tau-\sigma} Bu(\sigma) d\sigma \in X.
\]
More generally, given $(A,B,C)$ there is a satisfactory characterization of the fact that $A,B,C$ respectively are the infinitesimal generator, control operator and observation operator of some abstract linear system $\Sigma$ \cite[Theorem 5.1]{curtain_weiss_triple}. Note that the system $\Sigma$ is not uniquely determined by $A,B,C$ (one can add to $\mathbb{F}$ any $D \in \mathcal{L}_c(U;Y)$) and may fail to be regular. We are not aware of a nontrivial characterization of the fact that a quadruple of operators $(A,B,C,D)$ consists in the generating operators of a regular linear system $\Sigma$. 
\subsection{Cascade coupled systems}\label{sec:WP}
Now, we study the well-posedness of the cascade coupling of an abstract linear system and an abstract linear control system, having in mind \eqref{eq:cascade}. Let $\Sigma_c = ((\mathbb{T}_t)_{t \geq 0},(\Phi_t)_{t \geq 0},\mathbb{L}, \mathbb{F})$ be an abstract linear system (the controller) on the state space $Z$, the input space $U$ and the output space $Y$. Let $\Sigma_p = ((\mathbb{S}_t)_{t \geq 0},(\Psi_t)_{t \geq 0})$ be an abstract linear control system (the plant) on the state space $W$ and the input space $Y$. We would like to define the cascade system $\Sigma_{casc}$, formally for the moment, as the linear control system with state variable $\mathbf{Z}(t) = (z(t),w(t))$ satisfying
\begin{equation}\label{eq:abstract_cascade_state}
    \left\lbrace \begin{array}{rcl}
\dot{z}(t) &=& Az(t) + Bu(t), \\
z(0) &=& z^0, \\
y(t) &=& Cz(t) + Du(t), \\
\dot{w}(t) &=& Ew(t) + Fy(t), \\
w(0) &=& w^0,
\end{array}\right.
\end{equation}
where $(A,B,C,D)$ are the generating operators of $\Sigma_c$ and $(E,F)$ are the generating operators of $\Sigma_p$. Our goal is to explain in which sense one has to understand \eqref{eq:abstract_cascade_state} in order to obtain a well-posed linear system. Formal computations yield 
\begin{equation}\label{eq:def_abstract_cascade_operator}
    \mathbf{Z}(t) = \left( \begin{array}{cc}
\mathbb{T}_t & 0 \\
\Psi_t \mathbb{L} & \mathbb{S}_t
\end{array} \right) \left( \begin{array}{c}
z^0 \\
w^0
\end{array} \right) + \left( \begin{array}{c}
\Phi_t\\
\Psi_t \mathbb{F}
\end{array} \right)u,
\end{equation}
and 
\begin{equation}\label{eq:formal_generator_casc}
    \dot{\mathbf{Z}}(t) = \left( \begin{array}{cc}
A & 0 \\
FC & E
\end{array} \right)\left( \begin{array}{c}
z^0 \\
w^0
\end{array} \right) + \left( \begin{array}{c}
B \\
FD
\end{array} \right)u(t). 
\end{equation}
The formula \eqref{eq:def_abstract_cascade_operator} already defines an abstract linear control system. 
\begin{Theorem}\label{theo:abstract_control_system}
Let $\Sigma_c :=((\mathbb{T}_t)_{t \geq 0},(\Phi_t)_{t \geq 0},\mathbb{L}, \mathbb{F})$ be an abstract linear system on $Z,U,Y$ and $\Sigma_p := ((\mathbb{S}_t)_{t \geq 0},(\Psi_t)_{t \geq 0})$ be an abstract linear control system on $W,Y$. Then, the operators 
\begin{equation}\label{eq:def_cascade}
\mathbb{G}_t := \left( \begin{array}{cc}
\mathbb{T}_t & 0 \\
\Psi_t \mathbb{L} & \mathbb{S}_t
\end{array} \right) ,\quad \chi_t :=  \left( \begin{array}{c}
\Phi_t\\
\Psi_t \mathbb{F}
\end{array} \right),
\end{equation}
are such that $\Sigma_{casc} := ((\mathbb{G}_t)_{t \geq 0},(\chi_t)_{t \geq 0})$ defines an abstract linear control system on the state space $Z \times W$ and the input space $U$. The generator $\mathcal{A}$ of $\mathbb{G}_t$ satisfies 
\[
\mathcal{A} = \left( \begin{array}{cc}
A & 0 \\
FC & E_{-1}
\end{array} \right),\quad D(\mathcal{A}) = \left\lbrace \left( \begin{array}{c}
z^0 \\
w^0
\end{array} \right) \in D(A) \times W : E_{-1}w^0 + FCz^0 \in W  \right\rbrace,
\]
where $E_{-1} : W \to D(E^*)'$ is the unique extension of $E : D(E) \to W$.
\end{Theorem}
\begin{proof}
Let us first verify that $\mathbb{G}_t$ is a $C_0$-semigroup on $\mathcal{X} := Z\times W$. For fixed $0 \leq t < \infty$, $\mathbb{G}_t$ is obviously a bounded operator on $\mathcal{X}$, being the identity for $t = 0$. The semigroup property follows from the composition properties of $\Sigma_p$ and $\Sigma_c$. It remains to prove the continuity of $t \mapsto \mathbb{G}_t \mathbf{Z}^0$, for fixed $\mathbf{Z}^0 \in \mathcal{X}$. Let $\mathbf{Z}^0 = (z^0,w^0) \in Z \times W$ be fixed, we have 
\[
\mathbb{G}_t \mathbf{Z}^0 = \left( \begin{array}{c}
    \mathbb{T}_t z^0 \\
    \Psi_t \mathbb{L} z^0 + \mathbb{S}_tw^0 
\end{array}\right),
\]
hence it is enough to verify that the controlled trajectory $t \mapsto \Psi_t \mathbb{L} z^0$ is continuous $[0,\infty) \to W$. The latter property is true when $\mathbb{L} z^0$ is replaced by any $y \in L^2_{\oploc}([0,\infty);Y)$, as shown in \cite[Proposition 2.3]{weiss_admissibility_control}, hence, $\mathbb{G}_t$ is indeed a $C_0$-semigroup on $\mathcal{X}$. 

For any time $0 \leq t < \infty$, the operator $\chi_t$ is clearly bounded $L^2_{\oploc}([0,\infty) ; U) \to \mathcal{X}$. The fact that $(\mathbb{G},\chi)$ satisfies the composition property \eqref{eq:composition} is trivial, so that it is indeed an abstract linear control system. 
\newline
\newline
Let $G$ be the infinitesimal generator of $\mathbb{G}_t$, we show that $G = \mathcal{A}$. Let $\mathbf{Z}^0 = (z^0,w^0) \in D(G)$, we have 
\begin{equation}\label{eq:convergence_domain}
    \left( \begin{array}{c}
    \frac{\mathbb{T}_t z^0 - z^0}{t} \\
    \frac{\Psi_t \mathbb{L} z^0 + \mathbb{S}_t w^0 - w^0}{t} 
\end{array}\right) = \frac{\mathbb{G}_t \mathbf{Z}^0 - \mathbf{Z}^0}{t} \xrightarrow[t \to 0^+]{\mathcal{X}} G \mathbf{Z}^0 =: \left( \begin{array}{c}
    G_1\mathbf{Z}^0  \\
    G_2 \mathbf{Z}^0
\end{array}\right). 
\end{equation}
Taking the first coordinate, we have that $(\mathbb{T}_t z^0 - z^0)/t$ converges in $Z$, as $t \to 0^+$, hence $z^0 \in D(A)$, and the latter limit equals $Az^0$. We deduce that $G_1\mathbf{Z}^0 = Az^0$. For the second coordinate, we observe that 
\[
\frac{\mathbb{S}_t w^0 - w^0 }{t} \xrightarrow[t \to 0^+]{W_{-1}} E w^0,\quad \frac{\Psi_t \mathbb{L} z^0 + \mathbb{S}_t w^0 - w^0}{t} \xrightarrow[t \to 0^+]{W} G_2 \mathbf{Z}^0.
\]
Moreover, it is elementary to verify that 
\begin{equation}\label{eq:trivial}
    \forall y \in C([0,\infty) ; Y),\quad \frac{1}{t} \Psi_t y \xrightarrow[t \to 0^+]{W_{-1}} Fy(0),
\end{equation}
hence with $y(t) := (\mathbb{L} z^0)(t)$, which is continuous as $z^0 \in D(A)$, we find
\[
\frac{\Psi_t \mathbb{L} z^0}{t} \xrightarrow[t \to 0^+]{W_{-1}} FCz^0. 
\]
Thus, taking the second coordinate in \eqref{eq:convergence_domain}, we find $FCz^0 + Ew^0 = G_2 \mathbf{Z}^0 \in W$. This shows that $G \subset \mathcal{A}$. For the converse inclusion, we let $\mathbf{Z}^0 = (z^0,w^0) \in D(\mathcal{A})$ and observe that $\mathbf{Z}^0 \in D(G)$ if and only if the term 
\[
\frac{\Psi_t \mathbb{L} z^0 + \mathbb{S}_t w^0 - w^0}{t}, 
\]
converges in $W$, as $t \to 0^+$. We compute 
\begin{align*}
    \frac{\Psi_t \mathbb{L} z^0 + \mathbb{S}_t w^0 - w^0}{t} &= \frac{1}{t} \int_0^t \mathbb{S}_{t-\sigma} FC \mathbb{T}_\sigma z^0 d\sigma + \frac{1}{t} \int_0^t \mathbb{S}_\sigma Ew^0 d\sigma \\
    &= \frac{1}{t} \int_0^t \mathbb{S}_\sigma FC \mathbb{T}_{t-\sigma}z^0 d\sigma + \frac{1}{t} \int_0^t \mathbb{S}_\sigma E w^0 d\sigma. \\
    &= \frac{1}{t} \int_0^t \mathbb{S}_\sigma (FCz^0 + Ew^0) d\sigma + \frac{1}{t} \int_0^t \mathbb{S}_\sigma FC (\mathbb{T}_{t-\sigma}z^0 - z^0) d\sigma. \numberthis{} \label{eq:awesome_decomposition}
\end{align*}
The first term in \eqref{eq:awesome_decomposition} converges to $FC z^0 + Ew^0$ in $W$, since this last element belongs to $W$, owing to the definition of $D(\mathcal{A})$. It remains to show that the second term of \eqref{eq:awesome_decomposition} vanishes as $t \to 0^+$, in $W$. For this, we write 
\[
\frac{1}{t} \int_0^t \mathbb{S}_\sigma FC (\mathbb{T}_{t-\sigma}z^0 - z^0) d\sigma = \frac{1}{t} \Psi_t y,\quad y(t) := C (\mathbb{T}_t z^0 - z^0),
\]
where the function $y$ is of class $H^1_{(0),\oploc}([0,\infty) ; Y)$, owing to $z^0 \in D(A)$ and the admissibility of $C$. The trajectory $\zeta : t \mapsto \Psi_t y$ is therefore of class $C^1([0,\infty);W)$ \cite[Lemma 4.2.8]{tucsnak2009observation} and satisfies $\dot{\zeta}(0) = \zeta(0) = 0$. By definition of the derivative, we have 
\[
\frac{1}{t} \Psi_t y = \frac{\zeta(t)-\zeta(0)}{t} \xrightarrow[t \to 0^+]{W} \dot{\zeta}(0) = 0.
\]
We deduce that \[
\frac{\Psi_t \mathbb{L} z^0 + \mathbb{S}_t w^0 - w^0}{t} \xrightarrow[t \to 0^+]{W} FC z^0 + Ew^0
\]
 This shows that $\mathcal{A} \subset G$, whence the equality. 
\end{proof}
The identification of the control operator $\mathcal{B} : U \to D(\mathcal{A}^*)'$ of $\Sigma_{casc}$ is slightly more technical. From the formal computation \eqref{eq:formal_generator_casc}, we guess that $\mathcal{B}$ should correspond to the matrix operator 
\[
M := \left( \begin{array}{c}
    B \\
    FD 
\end{array}\right),\quad M : U \to D(A^*)' \times D(E^*)'.
\]
However, because 
\[
D(\mathcal{A}^*) = \left\lbrace \left( \begin{array}{c}
    \varphi^0 \\
    \psi^0 
\end{array}\right) \in Z \times D(E^*) : (A^*)_{-1}\varphi^0 + C^*F^* \psi^0 \in Z \right\rbrace
\]
is in general \textit{not} a Cartesian product, the space $D(\mathcal{A}^*)'$ has no reason to be a Cartesian product, hence one should explain in what precise sense $\mathcal{B}$  is represented by $M$ defined above. This issue has been addressed in \cite[pp 27, 52-56]{weiss_feedback}, which deals with the more general situation of regular linear systems with feedbacks. Therein a very general procedure has been proposed to identify a space $\mathfrak{W}$ such that 
\[
\mathcal{X} \subset \mathfrak{W} \subset (D(A^*)' \times D(E^*)') \cap D(\mathcal{A}^*)',
\]
and $\mathcal{B} : U \to \mathfrak{W}$ is bounded. However, the notation $(D(A^*)' \times D(E^*)') \cap D(\mathcal{A}^*)'$ has no canonical meaning for us, and hides an identification \cite[eq. 7.10]{weiss_feedback}. We prefer to impose \textit{ad hoc} conditions allowing one to bypass  these non-canonical identifications. While this approach is less general and requires additional assumptions, it seems to us that it provides a more concrete treatment in the present setting.
\newline
\newline
To define $\mathcal{B}^*$ as an operator $D(\mathcal{A}^*) \to U$ using \eqref{eq:def_abstract_cascade_operator}, it is natural to define $B^*$ on the projection on the first coordinate of $D(\mathcal{A}^*)$, that is  
\begin{equation}\label{eq:def_cal_D}
\mathcal{D} := \{ \varphi^0 \in Z : \exists \psi^0 \in D(E^*),\quad (A^*)_{-1} \varphi^0 + C^*F^*\psi^0 \in Z \}.
\end{equation}
\begin{Definition}\label{def:natural_extension}
We say that the operator $B^*$, originally defined from $D(A^*)$ to $U$, defines an operator $\mathcal{D} \to U$ if the following holds: 
\begin{itemize}
\item $\mathcal{D}$ is endowed with a Banach space structure;
\item There exists a Banach space $V$ such that $D(A^*) \subset V \subset Z$ and $\mathcal{D} \subset V$ continuously and densely;
\item $B^*$ has a (unique) extension $V \to U$.
\end{itemize}
\end{Definition}
For simplicity we will continue to use the same symbols $B$ and $B^*$. 
\begin{Proposition}\label{prop:identification_control_operator}
In addition to the hypotheses of Theorem \ref{theo:abstract_control_system} above, we assume that $\Sigma_c$ is regular, with feedthrough operator $D$, and that $B^*$ naturally defines an operator $\mathcal{D} \to U$. Lastly, assume that 
\begin{equation}\label{eq:lim_zero_force}
\forall \mathrm{v} \in U,\quad \forall \varphi^0 \in \mathcal{D},\quad \langle \frac{1}{t} \Phi_t (1 \otimes \mathrm{v}) , \varphi^0 \rangle_Z \xrightarrow[t \to 0^+]{} \langle \mathrm{v} , B^* \varphi^0 \rangle_U.
\end{equation}
Then, 
\[
\forall \left( \begin{array}{c}
\varphi^0 \\
\psi^0
\end{array} \right) \in D(\mathcal{A}^*),\quad \mathcal{B}^* \left( \begin{array}{c}
\varphi^0 \\
\psi^0
\end{array} \right) = B^* \varphi^0 + D^*F^*\psi^0. 
\]
\end{Proposition}
The requirement \eqref{eq:lim_zero_force} is automatically satisfied when $\varphi^0 \in D(A^*)$. We do not know if \eqref{eq:lim_zero_force} always holds under the hypotheses of Proposition \ref{prop:identification_control_operator}.
\begin{Remark}
This result can be compared with those of \cite[\S 7]{weiss_feedback}, which deals with a more general situation but makes use of a non-canonical identification. See also \cite[Lemma 5.1]{curtain_weiss_dynamic}, which uses the same identification  as in  \cite{weiss_feedback}.
\end{Remark}
\begin{proof}
We identify the control operator  $\mathcal{B}$ of the abstract linear control system $(\mathbb{G},\chi)$, which is defined in \eqref{eq:def_cascade}. We fix $\mathrm{v} \in U$ and observe that 
\[
\left( \begin{array}{c}
\frac{1}{t} \Phi_t (1 \otimes \mathrm{v}) \\
\frac{1}{t} \Psi_t \mathbb{F} (1 \otimes \mathrm{v})
\end{array}\right)  = \frac{1}{t} \chi_t(1 \otimes \mathrm{v}) \xrightarrow[t \to 0^+]{\mathcal{X}_{-1}} \mathcal{B} \mathrm{v}.
\]
Let $(\varphi^0,\psi^0) \in D(\mathcal{A}^*)$, we have
\[
\left\langle \frac{1}{t} \chi_t(1 \otimes \mathrm{v}) , \left( \begin{array}{c}
\varphi^0\\ \psi^0
\end{array}\right) \right\rangle_{D(\mathcal{A}^*)',D(\mathcal{A}^*)} = \langle \frac{1}{t} \Phi_t (1 \otimes \mathrm{v}) , \varphi^0\rangle_Z + \langle \frac{1}{t} \Psi_t \mathbb{F} (1 \otimes \mathrm{v}) , \psi^0\rangle_W,
\]
hence it is enough to show that 
\begin{equation}\label{eq:conv_indentification}
\langle \frac{1}{t} \Phi_t (1 \otimes \mathrm{v}) , \varphi^0\rangle_Z \xrightarrow[t \to 0^+]{} \langle \mathrm{v}, B^* \varphi^0\rangle_U ,\quad \langle \frac{1}{t} \Psi_t \mathbb{F} (1 \otimes \mathrm{v}) , \psi^0\rangle_W \xrightarrow[t \to 0^+]{} \langle \mathrm{v} , D^*F^*\psi^0\rangle_{D(E^*)',D(E^*)}
\end{equation}
holds for all $\varphi^0 \in \mathcal{D}$ and $\psi^0 \in D(E^*)$. The convergence of the first term is granted by hypothesis, hence we shall focus only on the second pairing. We compute 
\begin{align*}
\frac{1}{t} \Psi_t \mathbb{F} (1 \otimes \mathrm{v}) &= \frac{1}{t} \int_0^t \mathbb{S}_{t-\sigma} F \mathbb{F}(1\otimes \mathrm{v})(\sigma)d\sigma \\
&= \frac{1}{t} \int_0^t F \mathbb{F}(1\otimes \mathrm{v})(\sigma)d\sigma + \frac{1}{t} \int_0^t (\mathbb{S}_{t-\sigma}-1)F\mathbb{F} (1\otimes \mathrm{v})(\sigma)d\sigma. 
\end{align*}
In the above right-hand side, we observe that one may pull $F$ out of the first integral, and by the regularity assumption \eqref{regularity}, the corresponding term converges to $FD \mathrm{v}$, in $W_{-1}$. We verify that the other term converges to $0$, weakly in $W_{-1} = D(E^*)'$: 
\[
\left\langle \frac{1}{t} \int_0^t (\mathbb{S}_{t-\sigma}-1)F\mathbb{F} (1\otimes \mathrm{v})(\sigma)d\sigma , \psi^0 \right \rangle_{D(E^*)',D(E^*)} = \frac{1}{t} \int_0^t \langle \mathbb{F}(1\otimes \mathrm{v})(\sigma) , F^*(\mathbb{S}_{t-\sigma}^*-1) \psi^0 \rangle_Y d\sigma,
\]
where $\mathbb{F}(1\otimes \mathrm{v}) \in L^2_{\oploc}([0,\infty);Y)$ (which implies that  $||\mathbb{F}(1\otimes \mathrm{v})||_{L^2([0,t);Y)} = O(\sqrt{t})$ as $t\to 0^+$) and $t \mapsto F^*(\mathbb{S}_t^*-1) \psi^0 \in H^1_{(0),\oploc}([0,\infty);Y) $ (which implies that  $||F^*(\mathbb{S}_{t-\cdot}^*-1) \psi^0||_{L^2([0,t);Y)}= O({t})$ as $t\to 0^+$) . This is enough to conclude to the second convergence in \eqref{eq:conv_indentification}, using Cauchy-Schwarz inequality.
\end{proof}
\subsection{Riesz basis property}
In this \S~we transfer certain types of Riesz basis properties from $A$ and $E$ to $\mathcal{A}$. 
\begin{Assumption}\label{ass:coupled_spectrum}
\begin{itemize}
\item[]
\item $A$ (resp. $E$) is a closed densely defined unbounded operator on $Z$ (resp. $W$). 
\item The operators $C,F$ are respectively bounded $D(A) \to Y$ and $Y \to D(E^*)'$.
\item The operators $A$ and $E$ are both diagonalizable in Riesz bases, respectively denoted $(z_j)$ and $(w_k)$. 
\item The operators $A$ and $E$ do not share any eigenvalue.
\end{itemize}
\end{Assumption}
We define a unbounded operator $\mathcal{A}$ on $\mathcal{X} := Z \times W$ exactly as in Theorem \ref{theo:abstract_control_system}. We denote $\lambda_1, \lambda_2,...$ (resp. $\mu_1,\mu_2,...$) the eigenvalues of $A$ (resp. $E$), possibly counted with multiplicities, in such a way that $z_j$ (resp. $w_k$) is associated to $\lambda_j$ (resp. $\mu_k$). Under the above assumptions, it is easy to show that $\sigma_p(\mathcal{A}) = \sigma_p(A) \cup \sigma_p(E)$, with eigenvectors given by 
\[
\mathbf{Z}_k := \left( \begin{array}{c}
0 \\
w_k
\end{array} \right),\quad \mathbf{Z}_j := \left( \begin{array}{c}
z_j \\
(\lambda_j - E)^{-1}FCz_j
\end{array} \right).
\]
The following is a sufficient condition for $\mathcal{A}$ to be diagonalizable in a Riesz basis. 
\begin{Proposition}\label{prop:riesz_basis}
The normalized eigenvectors of $\mathcal{A}$ form a Riesz basis of $\mathcal{X}$ whenever
\begin{equation}\label{eq:hyp_quad_close}
\sum_j \|(\lambda_j - E)^{-1}FCz_j \|_W^2 < \infty.
\end{equation}
\end{Proposition}
\begin{proof}
We aim at showing that the family 
\[
\mathcal{F} := \lbrace \tilde{\mathbf{Z}}_j : j = 1,2,... \rbrace \cup \lbrace \mathbf{Z}_k : k = 1,2,... \rbrace,\quad \tilde{\mathbf{Z}}_j := \frac{\mathbf{Z}_j}{\|\mathbf{Z}_j\|_\mathcal{X}}
\]
is a Riesz basis of $\mathcal{X}$. Note that the hypothesis \eqref{eq:hyp_quad_close} implies that $(\lambda_j - E)^{-1}FCz_j$ is bounded in $W$ with respect to $j$, and therefore $(\mathbf{Z}_j)_j$ is almost normalized in $\mathcal{X}$. We can therefore equivalently consider the family 
\[
\mathcal{F} := \lbrace \mathbf{Z}_j : j = 1,2,... \rbrace \cup \lbrace \mathbf{Z}_k : k =1,2,... \rbrace.
\]
By \cite[Chapter 6, Theorem 2.3]{intro_non_selfadj}, it is enough to show that $\mathcal{F}$ is $\omega$-independent and that it is quadratically close to another Riesz basis $\mathcal{G}$ of $\mathcal{X}$. Our reference Riesz basis will be 
\[
\mathcal{G} := \left\lbrace \left( \begin{array}{c}
z_j \\
0
\end{array} \right) : j = 1,2,... \right\rbrace \cup \left\lbrace \left( \begin{array}{c}
0 \\
w_k
\end{array} \right) : k = 1,2,... \right\rbrace,
\]
which is clearly a Riesz basis of $\mathcal{X}$. The fact that $\mathcal{F}$ and $\mathcal{G}$ are quadratically close is equivalent to the hypothesis \eqref{eq:hyp_quad_close}, hence we are left with the $\omega$-independence of $\mathcal{F}$. Because the family $\mathcal{F}$ is almost normalized, its $\omega$-independence can be rephrased as: for any sequence of coefficients $\alpha \in \ell^2(\mathcal{F})$, if the series $\sum \alpha_f f$ converges (for some summation order) to the zero vector, then, all the coefficients are zero. So let $(\alpha_f)$ be as such, write 
\[
\{\alpha_f\}_{f \in \mathcal{F}} = \{\beta_j\}_j \cup \{\gamma_k\}_k,
\]
and take the first component of the convergence 
\begin{equation}\label{eq:convergence_riesz}
    \sum \alpha_f f \xrightarrow[]{Z \times W} \left( \begin{array}{c}
0 \\
0  
\end{array} \right),
\end{equation}
(where we do not specify the summation order for simplicity) to obtain that 
\[
\sum \beta_j z_j = 0.
\]
Taking advantage of the $\omega$-independence of $(z_j)$ in $Z$, we obtain:
\[
\forall j = 1,2,...,\quad \beta_j = 0.
\]
Considering the second component of \eqref{eq:convergence_riesz}, we obtain:
\[
0 = \sum \beta_j (\lambda_j-E)^{-1}FCz_j  + \sum \gamma_k w_k = \sum \gamma_k w_k.
\]
Because the $(w_k)$ are $\omega$-independent, we obtain: 
\[
\forall k = 1,2,...,\quad \gamma_k = 0,
\]
which shows that $(\alpha_f)_{f \in \mathcal{F}}$ is identically zero. This concludes the proof. 
\end{proof}
\subsection{Applications to the heat-wave system}\label{sec:applications}
In this section we apply all the results of the previous section to the heat-wave system \eqref{eq:coupled_heat_wave}. 
\begin{proof}[Proof of Proposition \ref{prop:wp}]
Following the transposition method (see \cite[Chapter 2]{Coron} and \cite[Section 10.2]{tucsnak2009observation}), the equations \eqref{eq:heat} are well-posed in the Hadamard sense on the state space $Z = L^2(0,1)$ and the input space $U = \mathbb{C}$. More precisely, for all $0 < T < \infty$, $z^0 \in L^2(0,1)$ and $u \in L^2(0,T)$, there is a unique solution $z(\cdot) \in C([0,T] ; L^2(0,1))$ to \eqref{eq:heat}. It satisfies the bound 
\[
\|z(\cdot)\|_{C([0,T];Z)} \leq C(T) \left\lbrace \|z^0\|_Z + \|u\|_{L^2(0,T)}\right\rbrace. 
\]
Introduce the Neumann Laplacian as the unbounded operator $A$ on $L^2(0,1)$ with
\[
A = \partial_{xx},\quad D(A) = \{ z \in H^2(0,1) : z_x(0) = z_x(1) = 0 \}.
\]
It is known that $-A$ is diagonalizable in a Hilbert basis, with simple eigenvalues $\lambda_1 < \lambda_2 < ...$. We have
\[
\lambda_j = ((j-1) \pi)^2,\quad e_j(x) = \left\lbrace \begin{array}{cc}
    1, & j = 1, \\
    \sqrt{2} \cos((j-1)\pi x), & j \geq 2,
\end{array}\right. 
\]
where $e_j$ is a normalized eigenvector of $-A$ associated to $\lambda_j$. The solution $z(\cdot)$ of \eqref{eq:heat} then writes 
\[
z(t) = \sum_{j=1}^\infty e^{-\lambda_j t} \langle z^0,e_j \rangle e_j + \int_0^t \sum_{j=1}^\infty e^{-\lambda_j (t-\sigma)} B^*e_j u(\sigma)   e_j d\sigma,
 \]
where we have defined $B^* := \delta_1$. One sees that $z(\cdot)$ belongs to $C([0,T] \times [0,1])$, with norm controlled by that of $z^0$ and $u$, hence the output $y(t) := z(t,0)$ is well-defined. For $z^0 \in L^2(0,1)$ and $u \in L^2_{\oploc}[0,\infty)$ we denote $z[z^0,u](\cdot)$ the solution of \eqref{eq:heat}. We put 
\[
\mathbb{T}_t z^0 = z[z^0,0](t,\cdot),\quad \Phi_t u = z[0,u](t,\cdot),\quad (\mathbb{L} z^0)(t) = z[z^0,0](t,0),\quad (\mathbb{F} u) (t) = z[0,u](t,0),
\]
which obviously defines an abstract linear system, with output $Y = \mathbb{C}$. Its infinitesimal generator is $A$, its control operator is $B$ and the observation operator is $C := \delta_0$. We verify that it is regular with $D = 0$: for $\mathrm{v} \in \mathbb{C}$ and $t > 0$ we compute 
\[
\frac{1}{t}\int_0^t \mathbb{F}(1 \otimes \mathrm{v})(\sigma)d\sigma = \frac{1}{t} \int_0^t z[0,u](\sigma,0)d\sigma,
\] 
and observe that the function $t \mapsto z[0,u](t,0)$ is continuous and vanishes as $t = 0$, hence the regularity and $D = 0$.

A similar reasoning applies to the wave system \eqref{eq:waves}, it defines an abstract linear system on the state space $W$, the input space $Y$, and has the generating operators $(E,F)$, defined by
\[
W = H^1_{(1)}(0,1) \times L^2(0,1), \quad F^*  \left( \begin{array}{c}
\psi \\
\tilde{\psi}
\end{array}\right) = - \tilde{\psi}(0),\quad E = \left( \begin{array}{cc}
0 & 1 \\
\partial_{xx} & 0
\end{array} \right),
\]
and
\[
D(E) = \left\lbrace \left( \begin{array}{c}
w \\
\tilde{w}
\end{array}\right) \in H^2(0,1) \times H^1(0,1) : w(1) = \tilde{w}(1) = w_x(0) = 0 \right\rbrace.
\]

The first item of Theorem \ref{theo:abstract_control_system} applies, hence the cascade coupling defines an abstract linear control system. It is easy to verify that $D(\mathcal{A})$ is given by \eqref{eq:primal_domain}, and $D(\mathcal{A}^*)$ by \eqref{eq:adj_coupled}. The set $\mathcal{D}$, defined in \eqref{eq:def_cal_D}, is given here by 
\[
\mathcal{D} = \{ \varphi \in H^2(0,1) : \varphi_x(1) = 0 \},
\]
hence we put $V = H^{3/2}(0,1)$. That the inclusions $D(A^*) \subset V$ and $\mathcal{D} \subset V$ are dense follows from the density of $C_c^\infty(0,1)$ in $H^{1/2}(0,1)$ \cite[Chapter 1, Theorem 11.1]{LM1}. We are left to verify the hypothesis \eqref{eq:lim_zero_force}, for which we compute 
\[
\langle \frac{1}{t} \Phi_t (1 \otimes \mathrm{v}) , \varphi^0 \rangle_Z = \frac{\mathrm{v}}{t} \langle z(t), \varphi^0 \rangle_Z ,
\] 
where $z$ solves \eqref{eq:heat} with $z^0(x) \equiv 0$ and $u(t) \equiv 1$. From Duhamel's formula we have 
\[
z(t) = \int_0^t \sum_{j=1}^\infty e^{-\lambda_j(t-\sigma)} (B^*e_j)e_j d\sigma = t (B^* e_1)e_1 + \sum_{j=2}^\infty \frac{1-e^{-\lambda_j t}}{\lambda_j}(B^*e_j)e_j,
\]
so that 
\[
\frac{1}{t}\langle z(t) , \varphi^0 \rangle_Z = B^*e_1 \langle e_1,\varphi^0 \rangle_Z + \sum_{j=2}^\infty \frac{1-e^{-\lambda_j t}}{t\lambda_j}(B^*e_j) \langle e_j , \varphi^0 \rangle_Z.
\]
By integration by parts we see that 
\[
\forall \varphi^0 \in \mathcal{D},\quad \langle e_j, \varphi^0 \rangle_Z = O\left( \frac{1}{\lambda_j} \right),\quad j \to \infty,
\]
and from the inequality 
\[
\forall x > 0,\quad |1-e^{-x}| \leq x,
\]
we deduce that for all $j \geq 2$ and $0 < t < 1$ there holds 
\[
\left| \frac{1-e^{-\lambda_j t}}{t\lambda_j}(B^*e_j) \langle e_j , \varphi^0 \rangle_Z \right| \lesssim \frac{1}{\lambda_j}.
\]
By dominated convergence, we obtain 
\[
\frac{1}{t}\langle z(t) , \varphi^0 \rangle_Z \to B^*e_1 \langle e_1,\varphi^0 \rangle_Z + \sum_{j=2}^\infty (B^*e_j) \langle e_j , \varphi^0 \rangle_Z = B^*\varphi^0,
\]
as $t \to 0^+$, hence the conclusion. 
\end{proof}
\begin{proof}[Proof of Proposition \ref{prop:riesz_heat_wave}]
It can be easily checked that the heat and wave sub-systems \eqref{eq:heat} and \eqref{eq:waves} satisfy Assumptions \ref{ass:coupled_spectrum}. We now need to verify the summability condition \eqref{eq:hyp_quad_close} in order to apply Proposition \ref{prop:riesz_basis}. To do so, it is convenient to transform the coordinates $(w,\tilde{w})$ into the Riemann coordinates $(f,g)$ through 
\begin{equation}\label{eq:riemann}
f = \frac{\tilde{w} + w_x}{2},\quad g = \frac{\tilde{w}-w_x}{2},\quad f^0 = \frac{\tilde{w}^0 + w_x^0}{2},\quad g^0 = \frac{\tilde{w}^0-w_x^0}{2},
\end{equation}
We verify that this leaves invariant \eqref{eq:hyp_quad_close}: put $\mathcal{R} : W \to L^2(0,1) \times L^2(0,1)$ defined by $\mathcal{R}(w,\tilde{w}) = (f,g)$. Using the parallelogram identity one verifies that 
\[
\| (w,\tilde{w}) \|_W^2 = 2 \| \mathcal{R}(w,\tilde{w})\|_{L^2(0,1) \times L^2(0,1)}^2.
\]
If $\tilde{E} := \mathcal{R} E \mathcal{R}^{-1}$ and $\tilde{F} := \mathcal{R} F$, then working with the $(f,g)$ coordinates amounts to consider the control pair $(\tilde{E}, \tilde{F})$ in place of $(E,F)$. One readily computes 
\[
\|(\lambda_j - \tilde{E})^{-1} \tilde{F} Cz_j \|_{L^2(0,1) \times L^2(0,1)} = \|\mathcal{R}(\lambda_j - E)^{-1} F Cz_j \|_{L^2(0,1) \times L^2(0,1)} = \frac{1}{\sqrt{2}} \|(\lambda_j - E)^{-1} F Cz_j \|_W,
\]	 
hence we may proceed using the Riemann coordinates. 

Then, for all $j = 1,2,...$, we have
\begin{equation}\label{eq:Z_j^p}
\mathbf{Z}_j^p = \rho_j \left( \begin{array}{c}
e_j \\
f_j^p \\
g_j^p
\end{array} \right),
\end{equation}
where $\rho_j$ is a normalization constant, and $f_j^p$ and $g_j^p$ are defined by the system of equations
\[
\left\lbrace \begin{array}{rcl c}
\partial_x f_j^p(x) &=& \lambda_j f_j^p(x),&  0<x<1,\\
\partial_x g_j^p(x) &=& - \lambda_j g_j^p(x),& 0<x<1,\\
g_j^p(1) &=& - f_j^p(1), \\
g_j^p(0) - f_j^p(0) &=& 1.
\end{array}\right.
\]
The solution of the above system is
\[
f_j^p(x) = \frac{e^{\lambda_jx}}{1+ e^{2 \lambda_j}},\quad g_j^p(x) = - \frac{e^{2 \lambda_j}}{1+e^{2 \lambda_j}}e^{-\lambda_jx},
\]
and we observe that both terms converge to 0 as $j \rightarrow \infty$, pointwise on $(0,1)$ with an $L^\infty$ domination. By dominated convergence we deduce that the convergence is in $L^2(0,1)$, and noticing that $\rho_j$ in \eqref{eq:Z_j^p} can actually be taken as any almost normalization constant we can therefore set $\rho_j \equiv 1$. Thus, the condition \eqref{eq:hyp_quad_close} rephrases as 
\[
\sum_{j=1}^\infty \int_0^1 \left\lbrace \left| \frac{e^{\lambda_j x}}{1+e^{2 \lambda_j}} \right|^2 + \left| \frac{e^{2\lambda_j}}{1+e^{2 \lambda_j}} e^{-\lambda_jx}\right|^2 \right\rbrace dx < \infty
\]
which follows from 
\[
\int_0^1 \left| \frac{e^{\lambda_j x}}{1+e^{2 \lambda_j}} \right|^2 dx = \frac{e^{2 \lambda_j}-1}{2\lambda_j(1+e^{2\lambda_j})^2} \sim \frac{e^{-2 \lambda_j}}{2 \lambda_j},\quad j\rightarrow \infty,
\]
\[
\int_0^1   \left| \frac{e^{2\lambda_j}}{1+e^{2 \lambda_j}} e^{-\lambda_jx}\right|^2  dx = \left(\frac{e^{2 \lambda_j}}{1+e^{2 \lambda_j}}\right)^2 \frac{e^{-2\lambda_j}-1}{-2 \lambda_j} \sim \frac{1}{2\lambda_j},\quad j \rightarrow \infty.
\]
\end{proof}
The end of this section is devoted to the proof of Theorem \ref{theo:mixed_quakes}.  First of all, we will recall the strategy developed in \cite[Proposition 2.4]{lissy_douglas} at the abstract level of \S \ref{sec:WP}. Assume that the pair $(A,B)$ is null-controllable in arbitrarily small time. In this case, one can define the observation norm
\begin{equation}\label{eq:obs_norm}
\|\varphi^T\|_*^2 := \int_0^T \|B^*\varphi(t) \|_U^2 dt,
\end{equation}
for all $\varphi^T \in Z$, where $\varphi(\cdot)$ is the solution of 
\begin{equation}\label{eq:adjoint_A}
\left\lbrace \begin{array}{rcl c}
- \dot{\varphi}(t) &=& A^* \varphi(t), & 0 < t < T,\\
\varphi(T) &=& \varphi^T.
\end{array}\right.
\end{equation}
The quantity $\| \cdot \|_*$ is a norm \cite[\S 2.3.2]{Coron}, put $\hat{Z}$ any completion of $(Z,\|\cdot\|_*)$. Next observe that the adjoint system of \eqref{eq:cascade} is
\begin{equation}\label{eq:adjoint_abstract}
\left\lbrace \begin{array}{rcl c}
- \dot{\psi}(t) &=& E^*\psi(t),& 0 < t < T \\
\psi(T) &=& \psi^T, \\
-\dot{\varphi}(t) &=& A^*\varphi(t) + C^*F^*\psi(t), & 0 < t < T\\
\varphi(T) &=& \varphi^T, 
\end{array}\right.
\end{equation}
and define the observed signal by 
\[
\Gamma(\varphi^T, \psi^T)(t) := B^*\varphi(t),
\]
which defines a bounded operator $\Gamma : Z \times W \to L^2(0,T;U)$. The approximate controllability (at time $T$) of the system \eqref{eq:cascade} is equivalent to the injectivity of $\Gamma$. From the definition of $\hat{Z}$, the operator $\Gamma$ has a unique extension $\Gamma^e$ as a linear continuous operator from $\hat{Z} \times W$. From \cite[Proposition 2.4]{lissy_douglas}, a sufficient condition for the mixed controllability property \eqref{eq:contro_simul} to hold is that $\Gamma^e$ is injective. To apply this strategy, one subtlety is that when $\varphi^T \in \hat{Z}\setminus Z$,  there is no clear  meaning of  what is $\varphi(t)$ in \eqref{eq:adjoint_abstract}, even if by following \cite[Proof of Proposition 2.4]{lissy_douglas}, one can find an appropriate abstract extension to the map $\varphi^T \in Z\mapsto \varphi \in L^2(0,T;Z).$

To show that $\Gamma^e$ is injective, we will rather show that when $\varphi^T \in \hat{Z}$, the solution $\varphi(t)$ of \eqref{eq:adjoint_A} is of class $C([0,T);L^2(0,1))$ and satisfies the heat equation in a weak sense. 

\begin{Lemma}\label{lem:distrib_reg}
Assume that $(A,B)$ is null-controllable in arbitrarily small time and let $0 < T < \infty$. The map $\Xi : \varphi^T \mapsto \varphi(\cdot)$, which to any $\varphi^T \in Z$ associates the solution $\varphi(\cdot)$ of \eqref{eq:adjoint_A}, has a unique linear and continuous extension $ \hat{Z} \rightarrow  C([0,T);Z)$.
\end{Lemma}
$C([0,T);Z)$ is a Fréchet space with semi-norms $p_n(\varphi) = \sup_{0 < t < T-1/n} \|\varphi(t)\|_Z$.
\begin{proof}
Because the pair $(A,B)$ is null-controllable in arbitrarily small time, the pair $(A^*,B^*)$ is final state observable in time $T$ (see, \textit{e.g}, \cite[Theorem 11.2.1.]{tucsnak2009observation}). Therefore, for all $\tau \in [0,T)$, 
\[
\exists c > 0,\quad \forall \varphi^T \in Z,\quad \|\varphi(\tau)\|_Z^2 \leq c \int_\tau^T  \|B^*\varphi(t) \|_U^2 dt,
\]
where $\varphi(\cdot)$ is the solution of \eqref{eq:adjoint_A}. The latter inequality also implies that, for all $\tau \in [0,T)$,
\[
\exists c > 0,\quad \forall \varphi^T \in Z,\quad \|\varphi(\tau)\|_Z^2 \leq c \int_0^T  \|B^*\varphi(t) \|_U^2 dt.
\]
Denote $c(\tau)$ the smallest such constant, which is a non-decreasing function of $\tau$. Now fix $n \geq 1$, $\tau \in [0,T-1/n]$ and $\varphi^T \in Z$, we have 
\[
\| \varphi(\tau) \|_Z^2 \leq c(\tau) \int_0^T  \|B^*\varphi(t) \|_U^2 dt \leq c(T-1/n) \int_0^T  \|B^*\varphi(t) \|_U^2 dt = c(T-1/n) \| \varphi^T \|_*^2,
\]
hence the operator 
\[
\Xi_n : Z \rightarrow C([0,T-1/n];Z),\quad \Xi_n \varphi^T(t) = \varphi(t),
\]
has a unique linear and continuous extension $\Xi_n^e : \hat{Z} \rightarrow C([0,T-1/n];Z)$. Then, we observe that for all $\varphi^T \in Z$, the restriction of $\Xi_{n+1}\varphi^T$ to $[0,T-1/n]$ is $\Xi_n\varphi^T$. By continuity and density, this also holds when $\varphi^T \in \hat{Z}$, hence the maps $(\Xi_n^e)_{n=1}^\infty$ define a unique linear and continuous operator $\Xi^e : \hat{Z} \rightarrow C([0,T);Z)$. This operator is the unique extension of $\Xi$. 
\end{proof}
As a consequence, the map $\Theta : Z \times L^2(0,T;U)$, that  associates to $\varphi^T \in Z$ and $f \in L^2(0,T;Y)$ the solution $\varphi$ of 
\[
\left\lbrace \begin{array}{rcl c}
-\dot{\varphi}(t) &=& A^*\varphi(t) + C^*f(t),& 0 < t < T,\\
\varphi(T) &=& \varphi^T,
\end{array}\right.
\]
has a unique linear and continuous extension $\Theta^e : \hat{Z} \times L^2(0,T;U) \rightarrow C([0,T);Z)$. 
\begin{proof}[Proof of Theorem \ref{theo:mixed_quakes}]
Let $T \geq 2$, we show that $\Gamma^e$ is injective. To this aim let us first assume that $(\varphi^T , \psi^T) \in Z \times W$. In this case we have 
\[
\Gamma(\varphi^T,\psi^T)(t) = B^*\varphi(t) = \varphi(t,1),
\]
for almost every $t \in (0,T)$, where $(\varphi,\psi)$ solves the adjoint system \eqref{eq:adjoint_abstract}. In particular, with $f := F^*\psi$, we have that $\varphi = \Theta(\varphi^T,f)$ is a weak solution of
\begin{equation}\label{eq:heat_phi_smart}
\left\lbrace \begin{array}{rcl cc}
-\partial_t \varphi(t,x) &=& \partial_{xx}\varphi(t,x),& 0 < t < T, & 0 < x < 1, \\
\varphi_x(t,0) &=& f(t),\\
\varphi_x(t,1) &=& 0, \\
\varphi(t,1) &=& \Gamma(\varphi^T,\psi^T)(t),
\end{array}\right.
\end{equation}
in the sense that: for all $\zeta \in C_c^\infty((0,T) \times [0,1])$ with $\zeta_x(t,0) \equiv 0$, there holds 
\begin{equation}\label{eq:weak_phi}
0 = \int_0^T \int_0^1 \varphi (\zeta_t + \zeta_{xx}) dxdt - \int_0^T \left\lbrace \Gamma(\varphi^T,\psi^T)(t) \zeta_x(t,1) + \zeta(t,0)f(t) \right\rbrace dt.  
\end{equation}
Let us then verify that the weak formulation still holds when $\varphi^T \in \hat{Z}$: let $\varphi^T$ be as such and $\psi^T \in W$, fix a test function $\zeta$ and $\epsilon$ such that $0 < \epsilon < T$ and
\begin{equation}\label{eq:support_zeta}
\opsupp \zeta \subset [0,T-\epsilon] \times [0,1].
\end{equation}
Consider a a sequence $(\varphi^T_j$) approximating $\varphi^T$, for the $\hat{Z}$ norm, with $\varphi^T_j \in Z$ for all $j \in \mathbb{N}$. Denote $\varphi_j = \Theta(\varphi_j^T,f)$, for all $j \in \mathbb{N}$, we have 
\begin{align*}
0 &= \int_0^T \int_0^1 \varphi_j (\zeta_t + \zeta_{xx}) dxdt - \int_0^T \left\lbrace \Gamma(\varphi^T,\psi^T)(t) \zeta_x(t,1) + \zeta(t,0)f(t) \right\rbrace dt \\
&= \int_0^{T-\epsilon} \int_0^1 \varphi_j (\zeta_t + \zeta_{xx}) dxdt - \int_0^{T} \left\lbrace \Gamma(\varphi^T,\psi^T)(t) \zeta_x(t,1) + \zeta(t,0)f(t) \right\rbrace dt,
\end{align*}
and we may pass to the limit $j \to \infty$ because
\[
\varphi_j \xrightarrow[j \rightarrow \infty]{C([0,T-\epsilon] ; L^2(0,1))} \Theta^e(\varphi^T,f) =: \varphi, \quad \Gamma(\varphi^T_j,\psi^T) \xrightarrow[j \rightarrow \infty]{L^2(0,T)} \Gamma^e(\varphi^T,\psi^T),
\]
where the first limit follows from the boundedness of $\Theta^e$ (see Lemma \ref{lem:distrib_reg}). Therefore, $\varphi$ is a weak solution of \eqref{eq:heat_phi_smart} with $\Gamma^e$ in place of $\Gamma$. 

Assume now that $(\varphi^T,\psi^T) \in \opker \Gamma^e$. Introducing $\varphi := \Theta^e(\varphi^T,f)$ and $f := F^*\psi$, we have that
\begin{equation}\label{eq:weak_final}
0 = \int_0^T \int_0^1 \varphi (\zeta_t + \zeta_{xx}) dxdt - \int_0^T  \zeta(t,0)f(t) dt,  
\end{equation}
for all $\zeta \in C_c^\infty((0,T) \times [0,1])$ with $\zeta_x(t,0) \equiv 0$. We extend $\varphi \in C([0,T) ; L^2(0,1))$ by $0$ for $x > 1$, and we call $\varphi^e$ the extension. One easily sees that $\varphi^e \in C([0,T) ; L^2(0,\infty))$, and for all $\zeta \in C_c^\infty((0,T) \times (0,\infty))$ we have 
\[
\int_0^T \int_0^\infty \varphi^e (\zeta_t + \zeta_{xx}) dxdt = \int_0^T \int_0^1 \varphi (\zeta_t + \zeta_{xx}) dxdt = \int_0^T \int_0^1 \varphi (\zeta_t + \zeta_{xx}) dxdt - \int_0^T \zeta(t,0)f(t)  dt = 0.
\]
Thus, $\varphi^e$ solves $(-\partial_t - \partial_{xx})\varphi^e = 0$ in $\mathcal{D}'((0,T) \times (0,\infty))$. Up to a modification on a null set, the modified function still denoted by $\varphi^e$, we have $\varphi^e \in C^\infty((0,T) \times (0,\infty))$ and for all $0 < t < T$ the function $\varphi^e(t,\cdot)$ is analytic. Since $\varphi^e(t,x) = 0$ for almost every $(t,x) \in (0,T) \times (1,\infty)$, we deduce that $\varphi^e(t,x) = 0$ for almost all $(t,x) \in (0,T) \times (0,\infty)$. Returning to \eqref{eq:weak_final}, we deduce that $f(t) = 0$ for almost every $t \in (0,T)$. Then, on the one hand, because the wave equation \eqref{eq:waves} is approximately controllable in time $2$ and $T \geq 2$, we deduce that $\psi^T = 0$. On the other hand, 
\[
\| \varphi^T \|_*^2 = \int_0^T |\Gamma^e(\varphi^T,0)(t)|^2dt = \int_0^T |\Gamma^e(\varphi^T,\psi^T)(t)|^2dt = 0,
\]
where the first equality is the definition of $\| \cdot \|_*$. We conclude that $\varphi^T = 0$, as the norm $\| \cdot \|_*$ separates points in $\hat{Z}$.
\end{proof}
\section{Polynomial stabilization} \label{sec:pol}
\subsection{The Sylvester equation}\label{sec:sylv}
In this \S, we discuss the existence and properties of a solution of the Sylvester equation \eqref{eq:sylvester}. Assume that $a,b$ are generators of semigroups on a Banach space $\mathcal{B}$, and that $c$ is a bounded operator on $\mathcal{B}$. Then, if the growth bounds of $a$ and $b$ satisfy $\omega_0(a) + \omega_0(b) < 0$, the equation for $x \in \mathcal{L}_c(\mathcal{B})$
\[
ax + xb = -c,
\]
admits the unique solution 
\[
x = \int_0^\infty e^{ta} c e^{tb}dt,
\]
and the integral converges absolutely in $\mathcal{L}_c(\mathcal{B})$, see \cite[\S 9]{sylvester_survey} for a survey and \cite[Theorem 3]{phong} for a sharp result. Returning to \eqref{eq:sylvester}, we rewrite this equation as $(-E)\Pi + \Pi A = -FC$ and observe that the above-mentioned result should apply, at least formally, if $-E$ generates a semigroup (or equivalently, if $E$ generates a group).
\begin{Assumption}\label{ass:skew} \begin{itemize}
\item[]
\item $(A,B,C)$ are the generating operators of a regular abstract linear system with $D = 0$, on the state space $Z$, the input space $U$ and the output space $Y$.
\item $(E,F)$ are the generating operators of an abstract linear control system, on the state space $W$ and the input space $Y$. 
\item The operator $E$ is skew-adjoint.  
\end{itemize} 
\end{Assumption}
We emphasize that $E$ being skew-adjoint is a key hypothesis. It notably ensures that $D(E^*)=D(E)$ and $\omega_0(E)=0$, which allows us to give  meaning to the computations below.

It is convenient to introduce the following notion of weak solution.
\begin{Definition}
A weak solution of \eqref{eq:sylvester} is a bounded operator $\Pi : Z \rightarrow W$ such that for all $z \in D(A)$ and $w \in D(E)$, there holds
\[
\langle \Pi z , Ew \rangle_W + \langle \Pi A z,w \rangle_W = - \langle Cz,F^*w\rangle_Y.
\]
\end{Definition}
Because $E$ is skew-adjoint, the above is equivalent to: for all $z \in D(A)$, we have the equality 
\[
\Pi A z-E\Pi z  = -FCz \mbox{ in } W_{-1}.
\]
Our purpose is to exhibit and study a solution; we shall not deal with its uniqueness.\footnote{See \cite[\S 3.3, Proposition 1]{emirsajlow}.} Observe that the operator $FC$ is bounded $Z_1 \rightarrow W_{-1}$, hence we cannot apply the previously cited result \textit{verbatim}.  Denoting by $(\mathbb{S}_t)_{t \geq 0}$ the semigroup generated by $E$, and $E$ being skew-adjoint, we have that $(\mathbb{S}_t^*)_{t \geq 0}$ is the semigroup generated by $E^* = -E$. The equation \eqref{eq:sylvester} is then equivalent to $E^* \Pi + \Pi A = -FC$, for which the \textit{ansatz} is
\begin{equation}\label{eq:ansatz}
    \Pi = \int_0^\infty \mathbb{S}_t^* FC \mathbb{T}_t dt,
\end{equation}
where we recall that $(\mathbb{T}_t)_{t \geq 0}$ is the semigroup generated by $A$. The above integrand is a continuous function of $t$, taking values in the Banach space $\mathcal{L}_c(Z_1;W_{-1})$. The following is a sufficient condition for the above integral to converge in a certain sense in $\mathcal{L}_c(Z;W)$.
\begin{Proposition}\label{prop:sylv} Assume the Assumptions \ref{ass:skew} hold. Assume moreover that there exists a weight $\mathfrak{w} : (0,\infty) \rightarrow (0,\infty)$ and a constant $c > 0$ such that 
\begin{equation}\label{eq:cond_convergence}
    \int_0^\infty \| C\mathbb T_t z \|_Y^2 \mathfrak{w}(t) dt \leq c \| z \|_Z^2,\quad \int_0^\infty \| F^* \mathbb S_t w \|_Y^2 \frac{dt}{\mathfrak{w}(t)} \leq c \| w \|_W^2,
\end{equation}
holds for all $z \in Z$ and $w \in W$. Then, for all $z \in Z$ and $w \in W$, the integral 
\[
\pi(z,w) := \int_0^\infty \langle C\mathbb T_t z,F^* \mathbb S_t w\rangle_Y d t
\]
absolutely converges, and defines a continuous sesquilinear form $\pi : Z \times W \rightarrow \mathbb{C}$. The associated bounded operator $\Pi : Z \rightarrow W$ is a weak solution of \eqref{eq:sylvester}.
\end{Proposition}
\begin{proof}
We first show that the integral defining $\pi$ is convergent. Let $z \in Z$ and $w \in W$, from the Cauchy-Schwarz inequality, we have
\[
\int_0^\infty |\langle C\mathbb T_t z,F^* \mathbb S_tw\rangle_Y| d t \leq \left(\int_0^\infty \| C \mathbb{T}_t z \|_Y^2 \mathfrak{w}(t) dt \right)^{1/2} \left(\int_0^\infty \| F^* \mathbb{S}_t w \|_Y^2 \frac{dt}{\mathfrak{w}(t)} \right)^{1/2} \leq c \| z \|_Z \|w\|_W.
\]
Thus, the integral defining $\pi$ is absolutely convergent, defines a bounded sesquilinear form $\pi : Z \times W \rightarrow \mathbb{C}$, which defines a bounded linear operator $\Pi : Z \rightarrow W$ from the Riesz representation theorem. 

We now verify that $\Pi$ is a weak solution of \eqref{eq:sylvester}. We take $z \in D(A)$, $w \in D(E)$ and compute
\begin{align*}
\langle \Pi z,Ew\rangle_W + \langle \Pi A z,w\rangle_W &= \int_0^\infty \langle C\mathbb T_t z, F^*\mathbb S_t Ew\rangle_Y dt + \int_0^\infty \langle C\mathbb T_t Az, F^*\mathbb S_t w\rangle_Y dt\\
&= \int_0^\infty \frac{d}{dt}\langle C\mathbb T_t z, F^* \mathbb S_t w\rangle_Y dt\\
&=   \lim_{t \rightarrow \infty} \langle C\mathbb T_t z,F^* \mathbb S_t w\rangle_Y - \langle Cz,F^*w\rangle_Y.
\end{align*}
We conclude by noticing that the limit at the right-hand side vanishes, as the function whose limit is taken is $W^{1,1}(0,\infty)$.
\end{proof}
In all that follows, we consider $\Pi$ as constructed above. 
\subsection{Collocated feedback}\label{sec:colloc}
We consider the new coordinates
\begin{equation}
    \label{eq:change}
\begin{pmatrix}
z\\ p
\end{pmatrix} = \mathcal{T} \begin{pmatrix}
z\\ w
\end{pmatrix} := \begin{pmatrix}
1 & 0\\
\Pi & 1
\end{pmatrix}\begin{pmatrix}
z\\ w
\end{pmatrix}.
\end{equation}
$\Pi : Z \rightarrow W$ being bounded, the transformation $\mathcal{T} : Z \times W \rightarrow Z \times W$ is invertible, its inverse being given by
\[
\mathcal{T}^{-1} = \begin{pmatrix}
1 & 0\\
-\Pi & 1
\end{pmatrix}.
\]
In the new coordinates $(z,p)$, the equations of the abstract cascade coupled system \eqref{eq:cascade} (which are only formal for the moment) become \eqref{eq:p_w}, which we aim at stabilizing using the collocated feedback $u(t) = -(\Pi B)^*p(t)$. To make the latter well-defined, we make the following additional assumptions which are motivated by our forthcoming application of the results of \cite[\S 2.A]{chill}. 
\begin{Assumption}\label{ass:chill}
There exists a Banach space $\mathsf{Z}$ and a Hilbert space $\mathsf{W}$ such that the following assertions hold:
\begin{itemize}
    \item $Z \subset \mathsf{Z}$ and $\mathsf{W} \subset W$ with continuous dense inclusion. Put $\mathsf{W}'$ the dual space of $\mathsf{W}$ with respect to the pivot $W$.
    \item There exists a weight $\mathfrak{c} : (0,\infty) \rightarrow (0,\infty)$ and a constant $c > 0$ such that 
\[
    \int_0^\infty \| C\mathbb T_t z \|_Y^2 \mathfrak{c}(t) dt \leq c \| z \|_{\mathsf{Z}}^2,\quad \int_0^\infty \| F^* \mathbb S_t w \|_Y^2 \frac{dt}{\mathfrak{c}(t)} \leq c \| w \|_{\mathsf{W}}^2,
\]
holds for all $z \in D(A)$ and $w \in D(E)$.
\item $\opIm B \subset \mathsf{Z}$
\item $\Pi B \in \mathcal L_c(U,\mathsf{W}')$ is an admissible control operator for $E$. 
\end{itemize}
\end{Assumption}
The first two items impose that $\Pi$, defined in Proposition \ref{prop:sylv}, has a unique linear and continuous extension $\mathsf{Z} \rightarrow \mathsf{W}'$. The second item allows to define $\Pi B$ as a \textit{bona fide} composition, and we have 
\[
\Pi B (\Pi B)^* : \mathsf{W} \rightarrow \mathsf{W}'.
\]
Under the above hypotheses, from \cite[Lemma 2.2]{chill} the operator 
\[
E_\Pi = E_{-1} - \Pi B (\Pi B)^*,\quad D(E_\Pi) = \{ p \in \mathsf{W} : E_{-1}p - \Pi B (\Pi B)^*p \in W\},
\]
is dissipative and generates a $C_0$-semigroup on $W$. We define the generator of the closed-loop system, in the $(z,p)$ coordinates, as the operator $\mathfrak{A}$ defined on $Z \times W$ by
\[
\mathfrak{A} = \left( \begin{array}{cc}
A_{-1} & -B(\Pi B)^* \\
0 & E_\Pi
\end{array} \right),
\]
with domain
\[
D(\mathfrak{A}) = \{ (z,p) \in Z \times D(E_\Pi) : A_{-1}z -B(\Pi B)^*p \in Z\}.
\]
\begin{Proposition}\label{prop:closed_loop} Assume that the systems $(A,C)$ and $(E,F)$ satisfy the assumptions of Proposition \ref{prop:sylv}, and that  the assumptions \ref{ass:chill} are satisfied. Then, the operator $\mathfrak{A}$ defined above generates a $C_0$-semigroup on $Z \times W$. 
\end{Proposition}
\begin{proof}
Denote $(\mathbb{S}_\Pi(t))_{t\geq 0}$ the $C_0$-semigroup generated by $E_\Pi$ on $W$. To show that $\mathfrak{A}$ generates a $C_0$-semigroup on $Z \times W$, we apply the first two parts of Theorem \ref{theo:abstract_control_system} to the controller $(E_\Pi,0,-(\Pi B)^*,0)$ and the plant $(A,B)$. Among the hypotheses of these two items, the only non trivial one is that $(\Pi B)^*$ is an admissible observation operator for $E_\Pi$. Indeed, take $p^0 \in D(E_\Pi)$ and put $p(t) := \mathbb{S}_\Pi(t)p^0$. We have 
\[
\dot{p}(t) = E_\Pi p(t) = E p(t) - \Pi B (\Pi B)^*p(t),
\]
hence pairing against $p(t)$ for the scalar product of $W$ yields 
\begin{equation}\label{eq:PiB_admissible}
\frac{d}{dt} \frac{1}{2} \|p(t)\|_W^2 =   \opRe \langle \dot{p}(t) , p(t) \rangle_W =  -\|(\Pi B)^*p(t) \|_U^2.    
\end{equation}
Integrating over $(0,T)$, we obtain that $(\Pi B)^*$ is an admissible observation operator for $E_\Pi$.
\end{proof}
\subsection{Asymptotic stability}
In this subsection, we obtain general asymptotic stability results for the closed-loop system generated by $\mathfrak{A}$. It will be useful to consider the following observability notion.
\begin{Definition}
Consider an observation pair $(G,O)$, where $G$ generates a $C_0$-semigroup $(\mathbb{G}_t)_{t \geq 0}$ on some Hilbert space $H$. The pair $(G,O)$ is said to be approximately observable in infinite-time when the following holds
\[
\forall h^0 \in H,\quad O\mathbb{G}_t h^0 \equiv 0 \mbox{ on }(0,\infty) \Longrightarrow h^0 = 0.
\]
\end{Definition}
We shall work under the following additional hypotheses. 
\begin{Assumption}\label{ass:stabl}

\begin{itemize}
    \item[]
    \item The operator $E$ has compact resolvent.
    \item The operator $A$ is exponentially stable. 
    \item The pair $(E^*,F^*)$ is approximately observable in infinite-time.
    \item For all eigenvalue $i\mu$ of $E$, the operator $\mathbf{H}(i\mu)$ has dense range, where $\mathbf{H}$ stands for the transfer function of $(A,B,C)$.
\end{itemize}
\end{Assumption}
For the convenience of the reader we provide the definition of the transfer function of an abstract linear system $(\mathbb{T},\Phi,\mathbb{L},\mathbb{F})$. Given $u \in L^2_{\oploc}([0,\infty);U)$ we define the output $y(t) := (\mathbb{F}u)(t) \in L^2_{\oploc}([0,\infty) ; Y)$. By standard semigroup techniques one can show that if $u \in L^2([0,\infty);U)$, then $y$ is Laplace transformable, with abscissa of absolute convergence $\leq \omega_0(\mathbb{T})$. It can be further shown that there exists a holomorphic function $\mathbf{H} : \mathbb{C}_{\omega_0(\mathbb{T})} \to \mathcal{L}_c(U;Y)$ such that 
\[
\hat{y}(s) = \mathbf{H}(s) \hat{u}(s),\quad \opRe s > \max(\omega_0(\mathbb{T}),0),
\]
where $\mathbb{C}_\alpha := \{ s \in \mathbb{C} : \opRe s > \alpha \}$. The above relation further makes $\mathbf{H}$ unique, it is called the transfer function of $(\mathbb{T},\Phi,\mathbb{L},\mathbb{F})$. 
\begin{Remark}
The hypothesis that $\mathbf{H}(i\mu)$ has dense range for all $i \mu \in \sigma_p(E)$ is a non-resonance type condition. It already appeared in the context of the stabilization of cascade systems in \cite[Proposition 1]{marx2021forwarding} and \cite[Assumption 2.2]{paunonen2015controller}, in a different form.  
\end{Remark}
We first recall a few standard facts. Firstly, the assumption \ref{ass:chill} imposes that $\Pi B$ is an admissible control operator for $E$. $E$ being skew-adjoint, the operator $(\Pi B)^*$ is also an admissible observation operator for $E$. Because $E$ has compact resolvent, it is diagonalizable in a Hilbert basis, and its eigenvalues are of course purely imaginary. As shown in \cite[Proposition 6.9.1]{tucsnak2009observation}, the approximate observability in infinite-time of the pair $(E,(\Pi B)^*)$ is equivalent to the following condition: for every eigenvector $w$ of $E$, we have $(\Pi B)^*w \neq 0$. 
\begin{Proposition} \label{prop:non-resonance}
Assume that the hypotheses of Proposition \ref{prop:closed_loop}, together with Assumptions \ref{ass:stabl}, are all satisfied. Then, the pair $(E,(\Pi B)^*)$ is approximately observable in infinite-time.
\end{Proposition}

\begin{proof}
Assume by contradiction that there exists an eigenvector $w$ of $E$ such that $(\Pi B)^*w = 0$ and put $i \mu$ the eigenvalue associated to $w$. For all $u \in U$, we compute
\begin{align*}
    0 &= \langle u, (\Pi B)^* w  \rangle_U \\
    &= \langle \Pi B u , w \rangle_{W_{-1},W_1} \\
    &= \int_0^\infty \langle C \mathbb{T}_t B u , F^* \mathbb{S}_t w \rangle_Y dt  \\
    &= \int_0^\infty \langle C \mathbb{T}_t B u , F^* e^{i \mu t} w \rangle_Y dt \\
    &= \left\langle \int_0^\infty e^{-i\mu t} C \mathbb{T}_t B u dt , F^*w \right\rangle_Y \\
    &=  \langle \mathbf{H}(i\mu)u,F^*w \rangle_Y .
\end{align*}
In the above, the third equality is due to the fact that $\Pi$ assumes the same representation formula on $\mathsf{Z} \supset \opIm B$ (see the discussion after Assumption \ref{ass:chill}). The sixth equality is due to the definition of the transfer function. By density of the range of $\mathbf{H}(i\mu)$ we deduce that $F^*w = 0$, contradicting the approximate observability of $(E^*,F^*)$ in infinite time. 
\end{proof}

\begin{Theorem}\label{theo_strong_stbl}
Suppose that the assumptions of Proposition \ref{prop:non-resonance} are satisfied, that moreover $E_\Pi$ has compact resolvent, and that the inclusion $D(E_\Pi) \subset \mathsf{W}$ is continuous. Then, the semigroup generated by $\mathfrak{A}$ is strongly stable, \textit{i.e.}
\begin{equation}\label{eq:strong_stabl}
\forall (z^0,p^0) \in Z \times W,\quad \lim_{t\rightarrow + \infty} \Vert (z(t),p(t))\Vert_{Z\times W} = 0.
\end{equation}
\end{Theorem}

\begin{proof}
We begin by showing that the semigroup generated by $\mathfrak{A}$ is bounded. To this end, we take $(z^0,p^0) \in D(\mathfrak{A})$ and recall that $E_\Pi$ generates a semigroup of contractions, hence $p(t)$ is bounded with respect to $t$. Next, we integrate \eqref{eq:PiB_admissible} with respect to time to deduce 
\begin{equation}\label{eq:controle_obs}
    \int_0^\infty \| (\Pi B)^*p(t)\|_U^2 dt \leq \frac{1}{2} \|p^0\|_W^2.
\end{equation}
Thus, the input $u(t) := - (\Pi B)^*p(t)$ is $L^2([0,\infty);U)$. Since $z(t)$ satisfies $\dot{z}(t) = Az(t) + Bu(t)$, $A$ being  the generator of an exponentially stable semigroup, we deduce that $z(t)$ is bounded with respect to $t$, see, \textit{e.g.}, \cite[Proposition 4.4.5]{tucsnak2009observation}. The trajectory associated to $(z^0,p^0)$ is therefore bounded in $Z \times W$. It is clear that the bound is not greater than $K \|(z^0,p^0)\|_{Z \times W}$ for some constant $K > 0$ independent of $(z^0,p^0) \in D(\mathfrak{A})$. By continuity and density, we deduce that the trajectories emanating from $(z^0,p^0) \in Z \times W$ are also bounded.

From the previously shown fact and standard approximation arguments, we may take $(z^0,p^0) \in D(\mathfrak{A})$ in \eqref{eq:strong_stabl}. We fix $(z^0,p^0) \in D(\mathfrak{A})$, so that in particular $p^0 \in D(E_\Pi)$, and we first deal with the convergence of $p(t)$. Returning  to \eqref{eq:controle_obs}, we deduce that the function $t \mapsto (\Pi B)^*p(t)$ is of class $H^1([0,\infty);U)$, hence it must converge to 0 as $t \rightarrow \infty$. Then, consider $\omega(p^0)$ the $\omega$-limit set of the positive orbit $\lbrace p(t) : t\geq 0\rbrace$, that is
\[
\omega(p^0) := \left\lbrace p_\infty \in W : \exists t_n \uparrow \infty,\quad p(t_n) \xrightarrow[n \rightarrow \infty]{W} p_\infty \right\rbrace.
\]
We will use the LaSalle invariance principle: it is sufficient to show that $\omega(p^0) \subset \{ 0 \}$ and $\omega(p^0) \neq \emptyset$, which allows one to conclude that $p(t) \rightarrow 0$ as $t \rightarrow \infty$. Let us first prove that the set $\omega(p^0)$ is non-empty. Indeed, since $p^0 \in D(E_\Pi)$ and $E_\Pi$ generates a $C^0$-semigroup  of contractions on $ D(E_\Pi)$, for any sequence of positive  $\{t_n\}_{n\in\mathbb N^*}$ tending to $+\infty$, $p(t_n) \in D(E_\Pi)$ and $\{p(t_n)\}_{n\in\mathbb N^*}$ is bounded in $D(E_\Pi)$, by the previous discussion. Since $D(E_\Pi) \subset W$ is compact,  up to a subsequence, $\{p(t_n)\}_{n\in\mathbb N^*}$ converges weakly in $D(E_\Pi)$ and  strongly in $W$, and the set $\omega(p^0)$ is non-empty. Moreover, let $p_\infty \in \omega(p^0)$, for a sequence $t_n \uparrow \infty$ of positive times, we have that $p(t_n) \to p_\infty$ as $n \to \infty$, and the sequence $\{p(t_n)\}_{n\in\mathbb N^*}$ is bounded in $D(E_\Pi)$. Hence, by weak compactness, we have $p(t_n) \to p_\infty$ weakly in $D(E_\Pi)$. We deduce that $\omega(p^0) \subset D(E_\Pi)$. Moreover, clearly, $\omega(p^0)$ is invariant under the action of the semigroup generated by $E_\Pi$. In addition, we collect another important property of $\omega(p^0)$: any element $p_\infty \in \omega(p^0)$ is such that $(\Pi B)^* p_\infty = 0$. Indeed, let $p_\infty \in \omega(p^0)$, for a sequence $t_n \uparrow \infty$ of positive times we have that $p(t_n) \to p_\infty$ as $n \to \infty$, in the $W$-norm. Moreover, we have $(\Pi B)^* p(t) \to 0$ as $t \to \infty$ in the $U$-norm. As already mentioned, we also have $p(t_n) \rightharpoonup p_\infty$ in $D(E_\Pi)$. The inclusion $D(E_\Pi) \subset \mathsf{W}$ being continuous, we also have $p(t_n) \rightharpoonup p_\infty $ in $\mathsf{W}$. The operator $(\Pi B)^*$ is bounded $\mathsf{W} \rightarrow U$, so it is also weakly bounded and $(\Pi B)^*p_\infty = 0$. 

Now, let $q^0 \in \omega(p^0)$ and denote by $q(t)$ the solution of 
\[
\left\lbrace \begin{array}{rclc}
    \dot{q}(t) &=& E_\Pi q(t),& t \geq 0,  \\
    q(0) &=& q^0. 
\end{array}\right.
\]
Since $\omega(p^0)$ is invariant under the action of the semigroup generated by $E_\Pi$, we also have $(\Pi B)^*q(t) = 0$ for any $t\geqslant 0$. Thus, $q(t)$ satisfies the differential equation $\dot{q}(t) = E q(t)$. According to Proposition \ref{prop:non-resonance}, the pair $(E,(\Pi B)^*)$ is approximately observable in infinite-time, hence $q^0 = 0$. We therefore have $\omega(p^0) \subset \{ 0 \}$, hence $p(t) \rightarrow 0$ as $t \rightarrow \infty$.

Let us finally deal with $z(t)$. It satisfies 
\[
\left\lbrace \begin{array}{rcl}
    \dot{z}(t) &=& Az(t) + Bu(t),  \\
    z(0) & = & z^0, 
\end{array}\right. \quad  u(t) :=-(\Pi B)^* p(t),
\]
and because the operator $A$ generates an exponentially stable semigroup we may as well assume that $z^0 = 0$. The control $u$ has the properties 
\[
u \in L^2([0,\infty);U) \cap C([0,\infty);U),\quad u(t) \xrightarrow[t \to \infty]{U} 0,
\]
hence from \cite[Proposition 4.4.5]{tucsnak2009observation} we deduce that $z(t) \rightarrow 0$ as $t \rightarrow \infty$.
\end{proof}
The counterpart of the above result for the original system is as follows. By definition, the operator $\tilde{\mathcal{A}}$ defined algebraically by $\mathcal{T}^{-1} \mathfrak{A} \mathcal{T}$, with domain $D(\tilde{\mathcal{A}}) := \mathcal{T}^{-1} D(\mathfrak{A})$, is conjugate to $\mathfrak{A}$ as unbounded operators on $Z \times W$. It is therefore the generator of a $C_0$-semigroup on $Z \times W$, that is moreover strongly stable, under the hypotheses of Theorem \ref{theo_strong_stbl}. Then, we observe that 
\[
\tilde{\mathcal{A}} \left( \begin{array}{c}
z^0 \\ w^0 \end{array} \right) = \left( \begin{array}{cc}
    A_{-1} z^0 - B(\Pi B)^* (\Pi z^0 + w^0) \\
    -\Pi \left( A_{-1} z^0 - B(\Pi B)^* (\Pi z^0 + w^0) \right) + E_\Pi (\Pi z^0 + w^0)
\end{array}\right),
\]
with domain 
\[
D(\tilde{\mathcal{A}}) = \left\lbrace \left( \begin{array}{c}
z^0 \\ w^0 \end{array} \right) \in Z \times W : \Pi z^0 + w^0 \in D(E_\Pi),\quad A_{-1}z^0 - B(\Pi B)^*(\Pi z^0 + w^0) \in Z \right\rbrace.
\]
\textit{Formally}, the second coordinate of $\tilde{\mathcal{A}}(z^0,w^0)$ is computed as 
\begin{align*}
    -\Pi &\left( A_{-1} z^0 - B(\Pi B)^* (\Pi z^0 + w^0) \right) + E_\Pi (\Pi z^0 + w^0)  \\
    &~~= -\Pi A_{-1} z^0 + \Pi B(\Pi B)^* (\Pi z^0 + w^0)+ (E_{-1} - \Pi B(\Pi B)^*) (\Pi z^0 + w^0) \\
    &~~= -\Pi A_{-1} z^0 + E_{-1}  (\Pi z^0 + w^0) \\
    &~~= FC z^0 + E_{-1} w^0,
\end{align*}
as expected. However, the first equality is not justified as the operator $\Pi$ is bounded $Z \to W$ and $\mathsf{Z} \to \mathsf{W}'$, while $A_{-1} z^0 \in D(A^*)'$ and $B(\Pi B)^* (\Pi z^0 + w^0) \in \mathsf{Z}$, \textit{a priori}. For the same reason the third equality is not rigorous and requires an extension of $\Pi$. Moreover, the notation $FC z^0$ is tedious at this abstract level and would require a regularity property for $z^0$ or an extension of $FC$. For simplicity we do not provide a general result and refer the interested reader to the last Step of the proof of Theorem \ref{theo:stabl_heat_wave} for an example. 
\subsection{A polynomial rate}\label{ssub:pol}
In this section, we show that the feedback law of the previous abstract section allows to achieve polynomial stability for the concrete system \eqref{eq:coupled_heat_wave}. For heuristic purposes, we begin by performing formal computations at the abstract level of \S \ref{sec:colloc}. Through the change of coordinates\eqref{eq:change}, we are led to work in the $(z,p)$ coordinates, and to study $\mathfrak{A}$. Similarly as in the proof of Theorem \ref{theo_strong_stbl}, the point is to show stability for the operator $E_{\Pi}$. For the latter, a convenient criterion is the wave-packet condition, as presented in \cite[\S 3.A]{chill}. Let us assume that $A$ (resp. $E$) is diagonalizable in a Riesz basis $(e_j)$ (resp. $(f_k)$) with eigenvalues $(-\lambda_j)$ (resp. $(i\mu_k)$). We will also assume that $U = Y = \mathbb{C}$. From \cite[Theorem 3.5]{chill}, if one can bound below the quantities $|(\Pi B)^* f_k|$ for large $|k|$, then one deduces a upper bound for the operator norm of the resolvent of $E_{\Pi}$, and hence a non-uniform decay rate owing to \cite{borichev_tomilov}. Recalling the ansatz \eqref{eq:ansatz} we formally have
\[
(\Pi B)^* f_k = B^* \Pi^* \int_0^\infty \mathbb{T}_t^* C^* F^* \mathbb{S}_t dt f_k = \left (\int_0^\infty B^* \mathbb{T}_t^* C^* e^{i \mu_k t} dt \right ) F^*f_k. 
\]
We expect $|F^* f_k| \gtrsim 1$, so that only the integral in the right-hand side matters to estimate $|(\Pi B)^* f_k|$ from below. Now, still formally, 
\begin{equation}\label{eq:series_integ}
\int_0^\infty B^* \mathbb{T}_t^* C^* e^{i \mu_k t} dt = \int_0^\infty \sum_{j = 1}^\infty e^{-\lambda_j t} B^*e_j C e_j e^{i \mu_k t} dt = \sum_{j=1}^\infty \frac{B^*e_j C e_j}{i\mu_k - \lambda_j},
\end{equation}
which may admit a closed form and then be estimated as $|k| \to \infty$. 
\begin{Remark}
The series in the right-hand side of \eqref{eq:series_integ} is sensitive to perturbations. For instance, if one replaces $B^*e_1 C e_1,B^*e_2 C e_2...$ by $B^*e_1 C e_1+1,B^*e_2 C e_2...$, then the value of the series is modified by $1/(i\mu_k - \lambda_1)$, which is non-negligible as $|k| \to \infty$. In practice we are only able to take advantage of these computations if the latter series has a closed form, which explains why we do not state a general result. 
\end{Remark}
We now make the discussion rigorous, and start by defining the concept of feedback stabilization. 
\begin{Definition}\label{deff}
Let $(A,B)$ be a well-posed control pair on the state space $X$ and the input space $U$, and $K : D(K) \subset X \to U$ a unbounded operator. \begin{itemize}
    \item We say that $K$ is a feedback for $(A,B)$ if the part of $A+BK$ in $X$, which is defined as the unbounded operator $A_K := \left.( A_{-1} + BK )\right|_X$ on $X$ with 
\[
D(A_K) = \{ x^0 \in D(K) : A_{-1}x^0 + BKx^0 \in X \},\quad A_Kx^0 = A_{-1}x^0 + BKx^0,
\]
generates a $C_0$-semigroup on $X$. 
\item The feedback $K$ stabilizes $(A,B)$ (at some rate)\footnote{The rate can be understood as any stabilization concept defined in \S \ref{sec:intro_stbl}.} if the semigroup generated by $A_K$ is stable (at some rate). 
 \item If $K$ is a feedback for $(A,B)$ and $B^*$ naturally defines an operator $D(A_K^*) \to U$ (see Definition \ref{def:natural_extension}), we say that $K$ is admissible for $(A,B)$ if $B$ is an admissible control operator for $A_K$. 
\end{itemize}
\end{Definition}
The existence of the operator $\Pi$ constructed in \S \ref{sec:sylv} requires the exponential stability of $A$, which is not the case in \eqref{eq:heat} (because of the Neumann boundary conditions). To overcome this, we pre-stabilize the system by considering $u(t) = -\alpha z(t,1) + \tilde{u}(t)$ for $\alpha > 0$ fixed small enough and $\tilde{u}(t)$ the new control. We consider the unbounded operator $A_\alpha$ on $L^2(0,1)$ defined as 
\[
A_\alpha = \partial_{xx},\quad D(A_\alpha) = \{ z \in H^2(0,1) : z_x(0) = z_x(1) + \alpha z(1) = 0  \}.
\]
The operator $-A_\alpha$ is self-adjoint, bounded below by a positive constant and with compact resolvent, hence it has a sequence of eigenvalues $0<\lambda_{1,\alpha} \leq  \lambda_{2,\alpha} \leq ...$ counted with multiplicities, such that $\lambda_{j,\alpha} \to \infty$ for any fixed $\alpha > 0$ and as $j \to \infty$. Moreover, any choice of normalized eigenvectors gives a Hilbert basis of $L^2(0,1)$. Remark that $-A_\alpha$ now generates a semigroup that is exponentially stable in the absence of a control term. The following Lemma gives a quantification of the fact that, as $\alpha \to 0^+$, the eigenvalue $\lambda_{j,\alpha}$ approaches the $j$-th eigenvalue of the Neumann Laplacian, denoted $\lambda_{j,0} = [(j-1)\pi]^2$. 
\begin{Lemma}\label{lem:spectral_asymp}
The eigenvalues of $-A_\alpha$ satisfy 
\begin{equation}\label{eq:spectral_genuine}
\forall \alpha > 0,\quad \forall j \geq 1,\quad (j-1)^2\pi^2 < \lambda_{j,\alpha} < \left( j - \frac{1}{2}\right)^2 \pi^2.
\end{equation}
Moreover,
\[
 \sqrt{\lambda_{1,\alpha}} = \sqrt{\lambda_{1,0}} + \alpha^{1/2} - \frac{1}{6} \alpha^{3/2} + \frac{2}{45} \alpha^{5/2} +  O(\alpha^3),\quad \alpha \to 0^+.
 \]
and 
\[
\sqrt{\lambda_{j,\alpha}} = \sqrt{\lambda_{j,0}} + \frac{\alpha}{\sqrt{\lambda_{j,0}}} - \frac{\alpha^2}{\lambda_{j,0}^{3/2}} + \frac{o_j(\alpha^2)}{\lambda_{j,0}^{3/2}}, \quad \alpha \to 0^+,\quad j \geq 2,
\]
where the term $o_j(\alpha^2)$ depends on $j$ and $\alpha > 0$ but has an implicit constant that is uniform in $j \geq 2$ with respect to $\alpha \to 0^+$\footnote{To simplify, we will say that $o_j$ is uniform in $j$.}. 
\end{Lemma}
\begin{proof}
An elementary computation shows that if $\lambda > 0$ is an eigenvalue of $-A_\alpha$, then a corresponding eigenvector is $\cos(\sqrt{\lambda}x)$ and we have the equation 
\begin{equation}\label{eq:vp_robin}
    - \sqrt{\lambda} \sin(\sqrt{\lambda}) + \alpha \cos(\sqrt{\lambda}) = 0. 
\end{equation}
One readily verifies that $\sqrt{\lambda}$ cannot be an integer multiple of $\pi$, hence $(j-1) \pi < \sqrt{\lambda} < j\pi$ for some $j \geq 1$. By sign considerations, one further sees that $(j-1) \pi < \sqrt{\lambda} < (j-1/2)\pi$. Conversely, for all $j \geq 1$, the intermediate value theorem yields a unique solution to \eqref{eq:vp_robin} in $(j-1) \pi < \sqrt{\lambda} < (j-1/2)\pi$. Therefore, the eigenvalues of $-A_\alpha$ are simple and $(j-1) \pi < \sqrt{\lambda_{j,\alpha}} < (j-1/2)\pi$. 

We first show the asymptotics on $\sqrt{\lambda_{1,\alpha}}$. To match later notations, set $x_{1,\alpha}=\sqrt{\lambda_{1,\alpha}}$. We have $\alpha = f(x_{1,\alpha})$ with $f(x) = x \tan x$, which is a smooth bijection from $(0,\pi/2)$ to $(0,\infty)$. Thus $x_{1,\alpha} \to 0^+$ as $\alpha \to 0^+$. We have 
\[
x_{1,\alpha} = a\alpha^{1/2} +b\alpha^{3/2} +c \alpha^{5/2} + O(\alpha^3) \Longleftrightarrow \alpha = f\left( a\alpha^{1/2} +b\alpha^{3/2} +c \alpha^{5/2} + O(\alpha^3) \right),
\]
and since   
\[
f(x) = x^2 + \frac{1}{3}x^4 + \frac{2}{15} x^6 + O(x^8) ,\quad x \to 0^+,
\]
we deduce 
\[
f\left( a\alpha^{1/2} +b\alpha^{3/2} +c \alpha^{5/2} + O(\alpha^3) \right) = a^2 \alpha + \left(2ab + \frac{a^4}{3}\right)\alpha^2 + \left(2ac + \frac{4a^3b}{3} + \frac{2a^6}{15}\right) \alpha^3 + O(\alpha^{7/2}).
\]
The system 
\[
\left\lbrace \begin{array}{lcc}
a^2 &=& 1,\\
\displaystyle 2ab + \frac{a^4}{3} &=& 0, \\
\displaystyle 2ac + \frac{4a^3b}{3} + \frac{2a^6}{15} &=& 0, \\
\end{array}\right.
\]
has the solution $(a,b,c) = (1,-1/6,2/45)$, hence the asymptotics on $x_{1,\alpha}$.

Now fix $j \geq 2$, the equation \eqref{eq:vp_robin} rewrites in the new variable $x_{j,\alpha} := \sqrt{\lambda_{j,\alpha}} - (j-1)\pi \in (0,\pi/2)$ as
\[
\left( \frac{x_{j,\alpha}}{\sqrt{\lambda_{j,0}}} + 1 \right) \tan(x_{j,\alpha}) = \frac{\alpha}{\sqrt{\lambda_{j,0}}},
\]
owing to the $\pi$-periodicity of $\tan$. Put 
\[
f_j(x) := \left( \frac{x}{\sqrt{\lambda_{j,0}}} + 1 \right) \tan(x),
\]
which is smooth on $(0,\pi/2)$ and increases from $f_j(0^+) = 0$ to $f_j(\pi/2^-) = +\infty$, hence for fixed $j \geq 2$ we have $x_{j,\alpha} \to 0^+$ as $\alpha \to 0^+$. We adapt the above method, for fixed constants $a,b,c,d,$ we have
\begin{equation}\label{xjep}
x_{j,\alpha} = a \frac{\alpha}{\sqrt{\lambda_{j,0}}} + b \frac{\alpha^2}{\sqrt{\lambda_{j,0}}} + c \frac{\alpha^2}{\lambda_{j,0}} + d \frac{\alpha^2}{\lambda_{j,0}^{3/2}} + \frac{O_j(\alpha^3)}{\lambda_{j,0}^{3/2}},
\end{equation}
as $\alpha \to 0^+$ and uniformly in $j \geq 2$, if and only if 
\[
\frac{\alpha}{\sqrt{\lambda_{j,0}}} = f_j\left(  a \frac{\alpha}{\sqrt{\lambda_{j,0}}} + b \frac{\alpha^2}{\sqrt{\lambda_{j,0}}} + c \frac{\alpha^2}{\lambda_{j,0}} + d \frac{\alpha^2}{\lambda_{j,0}^{3/2}} + \frac{O_j(\alpha^3)}{\lambda_{j,0}^{3/2}} \right),
\]
as $\alpha \to 0^+$ and uniformly in $j \geq 2$. From the definition of $f_j$, we see that
\[
f_j(x) = \left( \frac{x}{\sqrt{\lambda_{j,0}}} + 1 \right)\left(
x + O(x^3)\right),
\qquad x \to 0^+,
\]
where the $O$ is independent on $j$,
so that
\[
f_j(x_{j,\alpha})
=
x_{j,\alpha}+\frac{x_{j,\alpha}^2}{\sqrt{\lambda_{j,0}}}
+O_j(x_{j,\alpha}^3),
\qquad  \alpha \to 0^+,
\]
where we used that $x_{j,\alpha}/\sqrt{\lambda_{j,0}}$ is bounded from above independently on $j$,  so that  $O_j$ is uniform in $j$ or $\alpha$.
Since $f_j(x_{j,\alpha})={\alpha}/{\sqrt{\lambda_{j,0}}}$, and using \eqref{xjep}, we see that the claimed expansion for $x_{j,\alpha}$ is valid with $(a,b,c,d) = (1,0,0,-1)$, for some $o_j$ that is uniform in $j$, as it can be easily verified, since developing properly the reminders in the expansions give terms of the form $O_j(\alpha^n)/(\sqrt{\lambda_{j,0}})^m$, with $O_j$ uniform in $j$, $n\geqslant 3$ and $m\geqslant 3$.
\end{proof}
For all $j \geq 1$, a normalized eigenvector of $-A_\alpha$ associated to the eigenvalue $\lambda_{j,\alpha}$ is given by
\begin{equation}\label{eq:eigen_robin}
e_{j,\alpha}(x) = c_{j,\alpha} \cos(\sqrt{\lambda_{j,\alpha}} x),\quad c_{j,\alpha}^2 = \frac{2}{1+\displaystyle \frac{\sin^2(\sqrt{\lambda_{j,\alpha}})}{\alpha}}.
\end{equation}
We also recall that the wave system \eqref{eq:waves} has state $\mathsf{w}(t) = (w(t),w_t(t))$ and generating operators $(E,F)$ which have been introduced in \S \ref{sec:applications}. The operator $E$ is skew-adjoint, diagonalizable in a Hilbert basis $(f_k)_{k=-\infty}^{+\infty}$, where each $f_k$ is associated to the eigenvalue $i \mu_k$ and
\[
\mu_k = \pi \left( k + \frac{1}{2} \right),\quad f_k = \frac{\sqrt{2}}{\mu_k} \left( 
\begin{array}{c}
    \cos(\mu_k x) \\
   i \mu_k \cos(\mu_k x)
\end{array}
\right).
\]
\begin{proof}[Proof of Theorem \ref{theo:stabl_heat_wave}] \underline{Step 1:} We make the cascade coupling of $(A_\alpha, B, C,0)$ and $(E,F)$. 
\newline
\newline
We take $B^* = \delta_1$ and $C = \delta_0$, as for the heat equation \eqref{eq:heat}. We recall \eqref{eq:spectral_genuine} and observe that, for all $\alpha > 0$, the observations $B^*e_{j,\alpha}$ and $Ce_{j,\alpha}$ are bounded with respect to $j \geq 1$. A straightforward adaptation of the proof of Proposition \ref{prop:wp} shows that $(A_\alpha, B,C,E,F)$ satisfies the assumptions of Theorem \ref{theo:abstract_control_system} and Proposition \ref{prop:identification_control_operator}. The cascade coupled system is generated by $(\mathcal{A}_\alpha,\mathcal{B})$, where 
\[
D(\mathcal{A}_\alpha) =  \left\lbrace \left( \begin{array}{c}
z \\ w \\ \tilde w
\end{array} \right) \in H^2(0,1) \times H^2(0,1) \times H^1(0,1) : \left\vert \begin{array}{lll cc}
z_x(0) &=& z_x(1) + \alpha z(1) &=& 0, \\
w(1) &=&\tilde w(1) &=&0, \\
w_x(0) &=& z(0),
\end{array}\right. \right\rbrace.
\]
and
\[
\mathcal{A}_\alpha\left( \begin{array}{c}
z \\
w \\
\tilde{w}
\end{array} \right) = \left( \begin{array}{c}
z_{xx} \\
\tilde{w} \\
w_{xx}
\end{array} \right).
\]
\underline{Step 2:} We collect the results of \S \ref{sec:pol} for the system $(A_\alpha,B,C,E,F)$. \newline
\newline
From Proposition \ref{prop:sylv}, the Sylvester equation 
\[
(-E)\Pi + \Pi A_\alpha = -FC
\]
has (at least) one weak solution, denoted $\Pi_\alpha$, which is bounded $Z \to W$. We claim that $\Pi_\alpha$ has a unique bounded extension $Z_{-1/2,1} \to W$, where the space $Z_{-1/2,1}$ is an extrapolation space of $Z$ defined by the finiteness of
\[
\left\| \sum_{j=1}^\infty z_j e_{j,\alpha} \right\|_{Z_{-1/2,1}} = \sum_{j=1}^\infty \frac{|z_j|}{1+ \sqrt{\lambda_j}}.
\]
Let us show it, by adapting the proof of Proposition \ref{prop:sylv}. We put $\mathfrak{w}(t) := e^{\epsilon t}$ for some fixed $0 < \epsilon < 2\lambda_{1,\alpha}$. The pair $(E,F)$ is well-posed and $\mathbb{S}_t$ is $4$-periodic, hence the second estimate in \eqref{eq:cond_convergence} holds. We show that the first estimate holds with $\|z\|_Z$ replaced by $\| z \|_{Z_{-1/2,1}}$ in the right-hand side. To this end, let 
\[
z = \sum_{j=1}^\infty z_j e_{j,\alpha} \in Z,
\]
we compute 
\[
C e^{tA_\alpha} z = \sum_{j=1}^\infty e^{-\lambda_{j,\alpha}t} C e_{j,\alpha} z_j, \quad \sup_{j \geq 1} |C e_{j,\alpha} | < \infty,
\]
hence the triangular inequality for series brings 
\begin{equation}\label{eq:meilleur_estimee_analytic}
\| C e^{tA_\alpha}z  \|_{L^2(0,\infty, \mathfrak{w}(t) dt)} \leq \sum_{j=1}^\infty \frac{|z_j|}{\sqrt{2\lambda_{j,\alpha}-\epsilon}} \lesssim \sum_{j=1}^\infty \frac{|z_j|}{1+ \sqrt{\lambda_j}}  = \|z \|_{Z_{-1/2,1}},    
\end{equation}
whence the extension, by a straightforward adaptation of the proof of Proposition \ref{prop:sylv}. 

We consider now the new coordinates $(z,p)$ defined similarly as in \eqref{eq:change_variable}, with $\Pi_\alpha$ in place of $\Pi$. We define the operator $E_{\Pi_\alpha}$ on the state space $W$ by 
\[
E_{\Pi_\alpha} = E - \Pi_\alpha B (\Pi_\alpha B)^*,\quad D(E_{\Pi_\alpha}) = D(E).
\]
It generates a $C_0$-semigroup on $W$, being a bounded perturbation of $E$. We define the generator of the closed-loop system in the $(z,p)$ coordinates, induced by the feedback law $u(t) = -(\Pi_\alpha B)^* p(t)$. It is the operator $\mathfrak{A}_\alpha$ defined on $\mathcal{X}$ by, 
\[
\mathfrak{A}_\alpha = \left( \begin{array}{cc}
A_\alpha & -B(\Pi_\alpha B)^* \\
0 & E_{\Pi_\alpha}
\end{array} \right),\quad 
D(\mathfrak{A}_\alpha) = \{ (z,p) \in Z \times D(E) : A_\alpha z - B(\Pi_\alpha B)^*p \in Z\}. 
\]
The hypotheses of Proposition \ref{prop:closed_loop} are clearly matched by $(A_\alpha, B, C, E, F)$, hence $\mathfrak{A}_\alpha$ generates a $C_0$-semigroup on $\mathcal{X}$. 
\newline
\newline
\underline{Step 3}: We show that $E_{\Pi_\alpha}$ is polynomially stable at the rate $1/\sqrt{1+t}$, for small enough $\alpha$.
\newline
\newline
We reason as in the beginning of this subsection, all the written computations are now valid and we have 
 \[
(\Pi_\alpha B)^* f_k = \sqrt{2} \sum_{j=1}^\infty \frac{B^* e_{j,\alpha} C e_{j,\alpha}}{i \mu_k - \lambda_{j,\alpha}},
\]
where the interchange between the series and the integral in \eqref{eq:series_integ} is justified by Fubini's theorem, as the series in the above right-hand side is absolutely convergent. We write
\begin{align*}
\sum_{j=1}^\infty \frac{B^* e_{j,\alpha} C e_{j,\alpha}}{i \mu_k - \lambda_{j,\alpha}} &= \sum_{j=1}^\infty \frac{B^* e_{j,\alpha} C e_{j,\alpha}}{i \mu_k - \lambda_{j,0}} + \sum_{j=1}^\infty \frac{B^* e_{j,\alpha} C e_{j,\alpha} (\lambda_{j,0} - \lambda_{j,\alpha})}{(i \mu_k - \lambda_{j,\alpha})(i \mu_k - \lambda_{j,0})}
\end{align*} 
where the second term satisfies 
\[
\exists c > 0,\quad \forall k \in \mathbb{Z},\quad \forall 0 < \alpha <1,\quad 
\left| \sum_{j=1}^\infty \frac{B^* e_{j,\alpha} C e_{j,\alpha} (\lambda_{j,0} - \lambda_{j,\alpha})}{(i \mu_k - \lambda_{j,\alpha})(i \mu_k - \lambda_{j,0})} \right| \leq \frac{c}{\mu_k^2},
\]
and will therefore be neglected. Recall the variable $x_{j,\alpha} := \sqrt{\lambda_{j,\alpha}} - \sqrt{\lambda_{j,0}}$ introduced in Lemma \ref{lem:spectral_asymp}. With \eqref{eq:eigen_robin} we deduce
\[ 
B^* e_{j,\alpha} C e_{j,\alpha} = \frac{2(-1)^{j-1} \cos(x_{j,\alpha})}{1+\displaystyle\frac{\sin^2(x_{j,\alpha})}{\alpha}}.
\]
For $j=1$ we have 
\[
x_{1,\alpha} = \alpha^{1/2} - \frac{1}{6} \alpha^{3/2} + \frac{2}{45} \alpha^{5/2} + O(\alpha^3),
\]
hence 
\[
B^*e_{1,\alpha} Ce_{1,\alpha} = 1 - \frac{\alpha}{6} - \frac{7\alpha^2 }{180} + O( \alpha^{5/2}),\quad \alpha \to 0^+.
\]
For $j \geq 2$, we have 
\[
x_{j,\alpha} = \frac{\alpha}{\sqrt{\lambda_{j,0}}} - \frac{\alpha^2}{\lambda_{j,0}^{3/2}} + \frac{o_j(\alpha^2)}{\lambda_{j,0}^{3/2}},
\]
with $o_j(\alpha^2)$ uniform in $j$. With the same notation we compute, for $j \geq 2$, 
\[
B^* e_{j,\alpha} C e_{j,\alpha} = 2(-1)^{j-1} \left( 1 - \frac{\alpha}{\lambda_{j,0}} - \frac{\alpha^2}{2\lambda_{j,0}} + \frac{3\alpha^2}{\lambda_{j,0}^2} + \frac{o_j(\alpha^2)}{\lambda_{j,0}^2}\right). 
\]
Now, 
\begin{align*}
\sum_{j=1}^\infty \frac{B^* e_{j,\alpha} C e_{j,\alpha}}{i \mu_k - \lambda_{j,0}} &= \frac{B^*e_{1,\alpha} Ce_{1,\alpha}}{i\mu_k} + \sum_{j=2}^\infty \frac{B^* e_{j,\alpha} C e_{j,\alpha}}{i \mu_k - \lambda_{j,0}} \\
&= \sum_{j \in \mathbb{Z}} \frac{(-1)^j}{i\mu_k - (j\pi)^2} \\
&~~~~~- \frac{\alpha}{6i\mu_k} -\alpha \sum_{j \in \mathbb{Z}^*}^\infty \frac{(-1)^{j}}{(j\pi)^2[i\mu_k - (j\pi)^2]} \\
&~~~~~- \frac{7\alpha^2}{180 i \mu_k} - \frac{\alpha^2}{2} \sum_{j \in \mathbb{Z}^*}^\infty \frac{(-1)^{j}}{(j\pi)^2[i\mu_k - (j\pi)^2]} \\
&~~~~~+ 3 \alpha^2 \sum_{j \in \mathbb{Z}^*}^\infty \frac{(-1)^{j}}{(j\pi)^4[i\mu_k-(j\pi)^2]} \\
&~~~~~+ \frac{O(\alpha^{5/2})}{i\mu_k} + \sum_{j=2}^\infty \frac{2(-1)^{j-1}o_j(\alpha^2)}{\lambda_{j,0}^2(i\mu_k-\lambda_{j,0})},
\end{align*}
and we proceed by computing or estimating the terms in the right-hand side separately. From the residue formula we have 
\[
\sum_{j \in \mathbb{Z}} \frac{(-1)^j}{i\mu_k - (j\pi)^2} = -2\opRes\left( \frac{\pi \csc(\pi z)}{i\mu_k - (z\pi)^2},\sqrt{i\mu_k} \right) = O(e^{-\epsilon\sqrt{|k|}}),\quad |k| \to \infty,
\]
for some $\epsilon > 0$. Next, 
\begin{align*}
\sum_{j \in \mathbb{Z}^*}^\infty \frac{(-1)^{j}}{(j\pi)^2[i\mu_k - (j\pi)^2]} &= \frac{1}{i\mu_k} \sum_{j \in \mathbb{Z}^*}^\infty \frac{(-1)^{j}}{(j\pi)^2} + \frac{1}{i\mu_k} \sum_{j \in \mathbb{Z}^*}^\infty \frac{(-1)^{j}}{i\mu_k - (j\pi)^2} \\
&= -\frac{1}{6i\mu_k} + O\left(\frac{1}{\mu_k^2}\right),
\end{align*}
and similarly
\begin{align*}
\sum_{j \in \mathbb{Z}^*}^\infty \frac{(-1)^{j}}{(j\pi)^4[i\mu_k - (j\pi)^2]} &= \frac{1}{i\mu_k} \sum_{j \in \mathbb{Z}^*}^\infty \frac{(-1)^{j}}{(j\pi)^4} + \frac{1}{i\mu_k} \sum_{j \in \mathbb{Z}^*}^\infty \frac{(-1)^{j}}{(j\pi)^2[i\mu_k - (j\pi)^2]} \\
&= -\frac{7}{360 i\mu_k} +O\left(\frac{1}{\mu_k^2}\right).
\end{align*}
Finally, we have 
\[
\left| \sum_{j=2}^\infty \frac{2(-1)^{j-1}o_j(\alpha^2)}{\lambda_{j,0}^2(i\mu_k-\lambda_{j,0})} \right| \leq \frac{o(\alpha^2)}{|\mu_k|},
\]
a situation which we denote from now on
\[
\sum_{j=2}^\infty \frac{2(-1)^{j-1}o_j(\alpha^2)}{\lambda_{j,0}^2(i\mu_k-\lambda_{j,0})} = o(\alpha^2) \otimes O\left( \frac{1}{\mu_k} \right),\quad \alpha \to 0^+,\quad |k| \to \infty. 
\]
Combining all of the above, and with similar notations, we arrive to 
\[
\sum_{j=1}^\infty \frac{B^* e_{j,\alpha} C e_{j,\alpha}}{i \mu_k - \lambda_{j,\alpha}} = - \frac{\alpha^2}{72i\mu_k} + O\left( \frac{1}{\mu_k^2}\right) + O(\alpha) \otimes O\left( \frac{1}{\mu_k^2}\right) + o(\alpha^2) \otimes O\left( \frac{1}{\mu_k} \right),
\]
as $\alpha \to 0^+$ and $|k| \to \infty$. We fix $\alpha_* > 0$ small so that 
\[
\limsup_{|k| \to \infty} \left|o(\alpha_*^2) \otimes O\left( \frac{1}{\mu_k} \right)\right| \frac{|\mu_k|}{\alpha_*^2} < \frac{1}{72},
\]
which allows 
\[
\liminf_{|k| \to \infty} |\mu_k| \left| \sum_{j=1}^\infty \frac{B^* e_{j,\alpha_*} C e_{j,\alpha_*}}{i \mu_k - \lambda_{j,\alpha_*}} \right| > 0.
\]
We deduce that there exist constants $\lambda,K > 0$ such that 
\[
\forall |k| > K,\quad \left| \sum_{j=1}^\infty \frac{B^* e_{j,\alpha_*} C e_{j,\alpha_*}}{i \mu_k - \lambda_{j,\alpha_*}} \right| \geq \frac{\lambda}{|\mu_k|}. 
\]
We apply \cite[Theorem 3.5]{chill} with $\delta(s) \equiv \pi$, $\mu(s)$ bounded ($\Pi_\alpha B$ is bounded) and $\gamma(s) = 1/(1+|s|)$. We deduce that $i \mathbb{R} \subset \rho(E_{\Pi_\alpha})$ and that 
\[
\| (is - E_{\Pi_\alpha})^{-1}\|_{W \to W} \lesssim 1+|s|^2,\quad s \to \pm \infty.
\]
From \cite[Theorem 2.4]{borichev_tomilov} we therefore have that $E_{\Pi_\alpha}$ is polynomially stable, at rate $1/\sqrt{1+t}$, for $\alpha > 0$ small enough. 
\newline
\newline
\underline{Step 4:} We propagate the polynomial stability of $E_{\Pi_\alpha}$ to $\mathfrak{A}_\alpha$.
\newline
\newline
To do this we will again rely on \cite{borichev_tomilov} and bound above the operator norm of the resolvent of $\mathfrak{A}_\alpha$. The semigroup generated by $\mathfrak{A}_\alpha$ is bounded from Theorem \ref{theo_strong_stbl}. Let us verify that $i \mathbb{R} \subset \rho(\mathfrak{A}_\alpha)$: for $s \in \mathbb{R}$, $(f,g) \in Z \times W$ and $(z,p) \in D(\mathfrak{A}_\alpha)$ we observe that 
\[
(is - \mathfrak{A}_\alpha)(z,p) = (f,g) \Longleftrightarrow \left\lbrace \begin{array}{rcl}
    p & = & (is-E_{\Pi_\alpha})^{-1}g,  \\
    z & = &  (is-A_\alpha)^{-1}(f-B(\Pi_\alpha B)^*p).
\end{array}\right.
\]
Note that the resolvents of $E_{\Pi_\alpha}$ and of $A_\alpha$ do exist at $is$ since $i \mathbb{R} \subset \rho(E_{\Pi_\alpha})$ and $A_\alpha$ is exponentially stable. Thus, the above equivalence shows that $is \in \rho(\mathfrak{A}_\alpha)$, hence $i \mathbb{R} \subset \rho(\mathfrak{A}_\alpha)$. To bound above the resolvent of $\mathfrak{A}_\alpha$ we will show that 
\[
\|(z,p)\|_{Z \times W} \lesssim (1+|s|^2) \|(f,g)\|_{Z \times W},
\]
which is enough to conclude. As established earlier 
\[
\|p \|_W \lesssim (1+|s|^2) \|g\|_W,
\]
hence 
\[
\|z\|_Z \lesssim \|f\|_Z + \|(is-A_\alpha)^{-1}B(\Pi_\alpha B)^*p\|_Z.
\]
Observe that $(\Pi_\alpha B)^* : W \rightarrow \mathbb{C}$ is bounded. We claim that
\[
\sup_{s \in \mathbb{R}} \|(is-A_\alpha)^{-1}B\|_{U \to Z} < \infty.
\]
To see this, we pass to the adjoint and observe that since $A_\alpha$ is self-adjoint,
\[
\|(is-A_\alpha)^{-1}B\|_{U \to Z}
=
\|B^*(-is-A_\alpha)^{-1}\|_{Z \to U}.
\]
For $z \in Z$ written as $z = \sum_{j \ge 1} z_j e_{j,\alpha}$, we have
\[
(-is-A_\alpha)^{-1}z
=
\sum_{j=1}^\infty \frac{z_j}{-is-\lambda_{j,\alpha}}\, e_{j,\alpha}.
\]
Hence
\[
B^*(-is-A_\alpha)^{-1}z
=
\sum_{j=1}^\infty \frac{z_j\, B^*e_{j,\alpha}}{-is-\lambda_{j,\alpha}}.
\]
By the Cauchy--Schwarz inequality,
\[
\|B^*(-is-A_\alpha)^{-1}\|_{Z \to U}^2
\le
\sum_{j=1}^\infty
\frac{|B^*e_{j,\alpha}|^2}{s^2+\lambda_{j,\alpha}^2}.
\]
Since the sequence $(B^*e_{j,\alpha})_{j\ge1}$ is bounded and
$\lambda_{j,\alpha} \asymp j^2$, we obtain
\[
\sum_{j=1}^\infty
\frac{|B^*e_{j,\alpha}|^2}{s^2+\lambda_{j,\alpha}^2}
\leqslant\sum_{j=1}^\infty\frac{|B^*e_{j,\alpha}|^2}{\lambda_{j,\alpha}^2}
< \infty,
\]
uniformly for $s \in \mathbb{R}$. 

This proves the desired bound. From \cite[Theorem 2.4]{borichev_tomilov}, $\mathfrak{A}_\alpha$ is polynomially stable, at rate $1/\sqrt{1+t}$ (for small enough $\alpha$). 
\newline
\newline
\underline{Step 5}: We conclude polynomial stabilization of the original system \eqref{eq:coupled_heat_wave}.
\newline
\newline
Consider the map $\mathcal{T}_\alpha : \mathcal{X} \to \mathcal{X}$, defined as in \eqref{eq:change} with $\Pi_\alpha$ in place of $\Pi$. One easily verifies that it is an isomorphism, hence the unbounded operator $\tilde{\mathcal{A}}_\alpha := \mathcal{T}_\alpha^{-1} \mathfrak{A}_\alpha \mathcal{T}_\alpha$ on $\mathcal{X}$ defined by 
\[
\tilde{\mathcal{A}}_\alpha \left( \begin{array}{c}
z^0 \\ \mathsf{w}^0 \end{array} \right) = \left( \begin{array}{cc}
    (A_\alpha)_{-1} z^0 - B(\Pi_\alpha B)^* (\Pi_\alpha z^0 + \mathsf{w}^0) \\
    -\Pi_\alpha \left( (A_\alpha)_{-1} z^0 - B(\Pi_\alpha B)^* (\Pi_\alpha z^0 + \mathsf{w}^0) \right) + E_{\Pi_\alpha} (\Pi_\alpha z^0 + \mathsf{w}^0)
\end{array}\right),
\]
with domain 
\[
D(\tilde{\mathcal{A}}_\alpha) = \{ (z^0,\mathsf{w}^0) \in Z \times W : \Pi_\alpha z^0 + \mathsf{w}^0 \in D(E),\quad (A_\alpha)_{-1}z^0 - B(\Pi_\alpha B)^*(\Pi_\alpha z^0 + \mathsf{w}^0) \in Z \},
\]
generates a $C_0$-semigroup, which is stable at rate $1/\sqrt{1+t}$. One readily verifies that 
\[
D(\tilde{\mathcal{A}}_\alpha) = \left\lbrace (z^0,\mathsf{w}^0) \in H^2(0,1) \times W : \left| \begin{array}{rcl}
\Pi_\alpha z^0 + \mathsf{w}^0 &\in& D(E), \\
z_x^0(0) &=& 0, \\
z_x^0(1) + \alpha z^0(1) &=& -(\Pi_\alpha B)^*(\Pi_\alpha z^0 + \mathsf{w}^0)
\end{array}\right. \right\rbrace.
\]
Consequently, for $(z^0,\mathsf{w}^0) \in D(\tilde{\mathcal{A}}_\alpha)$ we observe that 
\[
(A_\alpha)_{-1} z^0 \in L^2(0,1)\subset Z_{-1/2,1},\quad B(\Pi_\alpha B)^* (\Pi_\alpha z^0 + \mathsf{w}^0) \in \opIm B \subset Z_{-1/2,1},
\]
and recall that $\Pi_\alpha$ is bounded $Z_{-1/2,1} \to W$. We deduce that for all $(z^0,\mathsf{w}^0) \in D(\tilde{\mathcal{A}}_\alpha)$, there holds 
\begin{align*}
-\Pi_\alpha &\left( (A_\alpha)_{-1} z^0 - B(\Pi_\alpha B)^* (\Pi_\alpha z^0 + \mathsf{w}^0) \right) + E_{\Pi_\alpha} (\Pi_\alpha z^0 + \mathsf{w}^0) \\
&~~= -\Pi_\alpha (A_\alpha)_{-1}z^0 + \Pi_\alpha B (\Pi_\alpha B)^*(\Pi_\alpha z^0 + \mathsf{w}^0) + E(\Pi_\alpha z^0 + \mathsf{w}^0) \\
&~~~~~~~- \Pi_\alpha B (\Pi_\alpha B)^*(\Pi_\alpha z^0 + \mathsf{w}^0) \\
&~~= -\Pi_\alpha (A_\alpha)_{-1}z^0+ E_{-1} \Pi_\alpha z^0 + E_{-1} \mathsf{w}^0.
\end{align*}
We observe that the equation 
\[
-\Pi_\alpha (A_\alpha)_{-1}z^0+ E_{-1} \Pi_\alpha z^0 = FCz^0,
\]
holds for all $z^0 \in D(A_\alpha)$, the equality being in $W_{-1}$. 
Since the operators
\[
z \mapsto \Pi_\alpha (A_\alpha)_{-1}z,
\qquad
z \mapsto E_{-1}\Pi_\alpha z,
\qquad
z \mapsto FCz
\]
extend continuously from $H^2(0,1)$ into $W_{-1}$, the above identity extends by continuity and therefore holds in $W_{-1}$ for every $z\in H^2(0,1)$.
 Noticing in addition that the actions of $A_\alpha$ and $A$ coincide as differential operators on $H^2(0,1)$, we deduce that
\[
\tilde{\mathcal{A}}_\alpha = \left( \begin{array}{cc}
A - B (\Pi_\alpha B)^* \Pi_\alpha & -B(\Pi_\alpha B)^* \\
FC & E_{-1}
\end{array} \right).
\]
We introduce the unbounded operator $\mathcal{K}_\alpha : D(\mathcal{K}_\alpha) \subset \mathcal{X} \to \mathcal{U}$ given by 
\[
\mathcal{K}_\alpha(z,\mathsf{w}) = - \alpha z(1) -(\Pi_\alpha B)^*(\Pi_\alpha z + \mathsf{w}),
\]
on the domain
\[
D(\mathcal{K}_\alpha) = \left\lbrace (z^0,\mathsf{w}^0) \in H^2(0,1) \times W : \left| \begin{array}{rcl}
\Pi_\alpha z^0 + \mathsf{w}^0 &\in& D(E), \\
z_x^0(0) &=& 0, 
\end{array}\right. \right\rbrace.
\]
On $D(\mathcal{K}_\alpha)$ we have 
\[
\mathcal{A}_{-1} = \left( \begin{array}{cc}
A  & 0 \\
FC & E_{-1}
\end{array} \right) : D(\mathcal{K}_\alpha) \to L^2(0,1) \times W_{-1},
\]
and we observe that $\mathcal{B}$ is bounded $\mathcal{U} \to D(A^*)' \times W$. We deduce that $D(\tilde{\mathcal{A}_\alpha}) = D(\mathcal{A}_{\mathcal{K}_\alpha})$ and that on the latter set $\tilde{\mathcal{A}_\alpha} = \mathcal{A}_{\mathcal{K}_\alpha}$.
\end{proof}

\appendix
\section{Proof of Proposition \ref{prop:zz_pour_nous}}\label{sec:app_proof_zz}
The wave equation \eqref{eq:waves} is not approximately controllable in any time $0 < T < 2$, as is easily shown by the characteristic method, hence \eqref{eq:coupled_heat_wave} is not approximately controllable in time $T < 2$. The approximate controllability in time $T = 2$ is proved using the duality between controllability and observability, following the proof of Theorem \ref{theo:mixed_quakes}. 
\newline
\newline
For the last item we proceed as follows: we introduce the subset $\mathcal{V}$ of $\mathcal{X}$ defined by 
\[
\mathcal{V} = \left\lbrace \mathbf{Z} \in \mathcal{X} : \mathbf{Z} = \sum_{j=0}^{+\infty} \alpha_j \mathbf{Z}_j^p + \sum_{k=-\infty}^{+\infty} \beta_k \mathbf{Z}_k^h,\quad  \sum_{j=0}^\infty |\alpha_j|^2 +  \sum_{k=-\infty}^{+\infty} \left| \beta_k \sqrt{1+|k|}e^{\sqrt{\pi |k|/2}} \right|^2   < \infty \right\rbrace,
\]
which is a Hilbert space when endowed with the norm 
\[
\left\| \sum_{j=0}^{+\infty} \alpha_j \mathbf{Z}_j^p + \sum_{k=-\infty}^{+\infty} \beta_k \mathbf{Z}_k^h \right\|_\mathcal{V}^2 := \sum_{j=0}^\infty |\alpha_j|^2 +  \sum_{k=-\infty}^{+\infty} \left| \beta_k \sqrt{1+|k|}e^{\sqrt{\pi |k|/2}} \right|^2.
\]
Null controllability of \eqref{eq:coupled_heat_wave} (at time $0 < T < \infty$) with controls in $L^2(0,T)$ and initial data in $\mathcal{V}$ is defined by 
\[
\forall \mathbf{Z}^0 \in \mathcal{V},\quad \exists u \in L^2(0,T),\quad \mathbf{Z}(T) = 0.
\]
By standard duality considerations, the above is equivalent to the following observability inequality
\begin{equation}\label{eq:obs_spec_coupl}
    \sum_{j = 0}^\infty e^{-2T \lambda_j}|\alpha_j|^2 + \sum_{k=-\infty}^{+\infty} \left| \beta_k \frac{e^{-\sqrt{\pi |k|/2}}}{\sqrt{1+|k|}} \right|^2 \leq C \int_0^T \left| \sum_{j = 0}^\infty \alpha_j e^{-\lambda_j t} \mathcal{B}^* \mathbf{\Phi}_j^p + \sum_{k = -\infty
}^{+\infty} \beta_k e^{i\mu_kt} \mathcal{B}^*\mathbf{\Phi}_k^h \right|^2 dt,
\end{equation}
where the constant $C$ is independent of $(\alpha , \beta) \in \ell^2(\mathbb{N}) \times \ell^2(\mathbb{Z})$ and $\{ \mathbf{\Phi}_j^p \}_{j \in \mathbb{N}} \cup \{ \mathbf{\Phi}_k^h \}_{k \in \mathbb{Z}}$ is the family bi-orthogonal to $\{ \mathbf{Z}_j^p \}_{j \in \mathbb{N}} \cup \{ \mathbf{Z}_k^h \}_{k \in \mathbb{Z}}$. We observe that one can equivalently prove \eqref{eq:obs_spec_coupl} only for finitely supported $\alpha$ and $\beta$. We compute 
\[
\mathbf{\Phi}_j^p = \left( \begin{array}{c}
    e_j  \\
    0 \\
    0
\end{array}\right)
\]
where $e_j$ is the $j$-th eigenvector of $A$. We thus have 
\[
\mathcal{B}^* \mathbf{\Phi}_j^p = e_j(0) = \left\lbrace \begin{array}{ccc}
    1, & j = 0, \\
    \sqrt{2}, & j \geq 1, 
\end{array}\right. \asymp 1,
\]
where $a_n \asymp b_n$ means that $c|a_n| \leq |b_n| \leq C |a_n|$ for some $0 < c < C < \infty$ and all $n \in \mathbb{N}$. We moreover have 
\[
\mathbf{\Phi}_k^h = c_k \left( \begin{array}{c}
    \varphi_k(x)  \\
    - \frac{\cos(\mu_kx)}{\mu_k} \\
    \cos(\mu_k x)
\end{array}\right),
\]
where $c_k$ is a normalization constant and $\varphi_k$ is the solution of 
\[
\left\lbrace \begin{array}{rcl cc}
    \partial_{xx} \varphi_k(x) &=& -i\mu_k \varphi_k(x),& 0 < x < 1,  \\
    \partial_x \varphi_k(1) &=& 0,\\
    \partial_x \varphi_k(0) &=& 1.
\end{array}\right.
\]
Straightforward computations show that 
\begin{equation}\label{eq:estimate_obs_hyp}
    \varphi_k(x) = \frac{e^{-2 \sqrt{-i\mu_k}} e^{\sqrt{-i\mu_k}x} + e^{-\sqrt{-i\mu_k}x}}{\sqrt{-i\mu_k}(e^{-2\sqrt{-i\mu_k}} - 1)}, \quad \| \varphi_k \|_{L^2(0,1)} \rightarrow 0,\quad \varphi_k(1) \asymp \frac{e^{-\sqrt{|k|\pi/2}}}{\sqrt{1+|k|}} \quad |k| \to \infty,
\end{equation}
hence $c_k \asymp 1$. Noticing that the asymptotics of $\varphi_k(1) \asymp \mathcal{B}^*\mathbf{\Phi}_k^h$ is precisely the weight appearing in the left hand side of \eqref{eq:obs_spec_coupl}, we see that the latter observability inequality is equivalent to 
\[
\sum_{j = 0}^\infty e^{-2T \lambda_j}|\alpha_j|^2 + \sum_{k=-\infty}^{+\infty} \left| \beta_k \right|^2 \leq C \int_0^T \left| \sum_{j = 0}^\infty \alpha_j e^{-\lambda_j t}  + \sum_{k = -\infty
}^{+\infty} \beta_k e^{i\mu_kt} \right|^2 dt.
\]
The above inequality is proved in \cite[Proposition 1.8]{chowdhury}, which shows the last item of Proposition \ref{prop:zz_pour_nous}. 
\newline
\newline
For the second item of Proposition \ref{prop:zz_pour_nous}, the proof is similar (and slightly easier) than the proof of \cite[Proposition 9]{davron_rough}, we sketch the argument. The null controllability (at time $0 < T < \infty$) of the system \eqref{eq:coupled_heat_wave} with initial data in $D(\mathcal{A}^N)$ and control laws in $(H^N(0,T))'$ is defined as 
\[
\forall \mathbf{Z}^0 \in D(\mathcal{A}^N),\quad \exists u \in (H^N(0,T))',\quad \mathbf{Z}(T) = 0.
\]
The fact that $\mathbf{Z}(T)$ makes sense, as an element of $D(\mathcal{A}^{*N})'$, for $u \in (H^N(0,T))'$, is shown in \cite[Theorem 4]{davron_rough}. The above null controllability property is easily shown \cite[Proposition 19]{davron_rough} to be equivalent to 
\[
\exists C > 0,\quad \forall \mathbf{Z} \in D(\mathcal{A}^{*N}),\quad \| e^{T \mathcal{A}^*} \mathbf{Z} \|_{D(\mathcal{A}^{N})'} \leq C \| \Phi_T^* \mathbf{Z} \|_{H^N(0,T)}. 
\]
Assume by contradiction that the above property holds and test the above inequality against $\mathbf{Z} = \mathbf{\Phi}_k^h$ to find 
\begin{align*}
    \frac{1}{|\mu_k|^N} &\lesssim \| \mathbf{\Phi}_k^h \|_{D(\mathcal{A}^N)'} \\
    &= \| e^{T \mathcal{A}^*}\mathbf{\Phi}_k^h  \|_{D(\mathcal{A}^{N})'} \\
    &\lesssim \| \Phi_T^* \mathbf{\Phi}_k^h \|_{H^N(0,T)} \\
    &\lesssim | \mathcal{B}^* \mathbf{\Phi}_k^h| |\mu_k|^N.
\end{align*}
Because the term $\mathcal{B}^* \mathbf{\Phi}_k^h$ converges to $0$ exponentially as $|k| \to \infty$, we arrive to a contradiction. 
\printbibliography
\end{document}